\documentclass[a4paper,11pt,reqno]{amsart}
\usepackage{amsmath}
\usepackage{amsfonts}
\usepackage[applemac]{inputenc}
\usepackage[T1]{fontenc}
\usepackage[colorlinks=true]{hyperref}
\usepackage{graphicx}
\usepackage[mathscr]{euscript}
\usepackage{verbatim,color}
\usepackage{amssymb,latexsym}
\usepackage{amsthm}
\usepackage{color}

\usepackage[lmargin=2.5 cm,rmargin=2.5 cm,tmargin=3.5cm,bmargin=2.5cm,paper=a4paper]{geometry}
\DeclareMathOperator{\curl}{curl}

\newtheorem{thm}{Theorem}[section]

\newtheorem{lem}[thm]{Lemma}

\newtheorem{assumption}[thm]{Assumption}

\newtheorem{proposition}[thm]{Proposition}

\theoremstyle{remark}
\newtheorem{rem}[thm]{Remark}

\newcommand{\Ab}{\mathbf{A}}
\newcommand{\Cb}{\mathbb{C}}
\newcommand{\Fb}{\mathbf{F}}

\newcommand{\R}{\mathbb{R}}

\numberwithin{equation}{section}

\title[G-L energy]{On the Ginzburg-Landau energy with a magnetic field vanishing along a curve}
\author[A. Kachmar]{Ayman Kachmar}
\address{Lebanese University, Department of Mathematics, Nabatieh, Lebanon}
\email{ayman.kashmar@gmail.com}
 \author[M. Nasrallah]{Marwa Nasrallah}
 \address{Lebanese International University, Beirut, Lebanon \newline \& Lebanese University, faculty of Sciences, Section IV, Bekaa, Lebanon}
\email{marwa.nasrallah@liu.edu.lb}

\begin{document}

\maketitle
\begin{abstract}
The energy of a type II superconductor placed in a strong non-uniform, smooth and  signed  magnetic field  is displayed via a universal reference function defined by means of  a simplified two dimensional Ginzburg-Landau functional. We study the asymptotic behavior of this functional in a specific asymptotic regime, thereby linking it to a one dimensional functional,  using methods  developed by Almog-Helffer and  Fournais-Helffer devoted to the analysis of surface superconductivity  in the presence of a uniform magnetic field. As a result, we obtain an asymptotic formula reminiscent of the one for the surface  superconductivity regime, where the zero set of the magnetic field plays the role of the superconductor's surface. 
\end{abstract}

\section{Introduction}

During the two past decades, the mathematics of superconductivity has
 been the subject of intense activity (see \cite{dGe} for the physical
 background). 
One common model used to describe the behavior of a 
superconductor is the Ginzburg-Landau functional 
involving a pair $(\psi,\bf A)$, 
where $\psi$ is a wave function (called the order parameter) and 
${\bf A}$ is a vector field (called the magnetic potential), 
both being defined  on an open set $\Omega\subset\R^2$.
 The functional is
\begin{equation}\label{GL-Energy}
\mathcal{E}(\psi,{\bf A})=\int_{\Omega}\left[ |(\nabla -i\kappa H{\bf A})\psi|^{2}-\kappa^{2}|\psi|^{2}+\frac{\kappa^{2}}{2}|\psi|^{4}                                  \right]dx+\kappa^{2}H^{2}\int_{\Omega}|{\curl} {\bf A}-B_0|^{2}dx.
\end{equation}
The quantity $|\psi|^2$ measures the density of superconducting electrons (so that $\psi=0$ defines the normal state); ${\curl }{\bf A}$ measures the induced magnetic field; the parameter $H$ measures the strength of the external magnetic field and the parameter $\kappa >0$ is a characteristic  of the superconducting material. The function $B_0$  is a given function and accounts for the profile of an external non-uniform magnetic field. We will assume that $B_0\in C^3(\overline{\Omega})$.

Of particular physical interest is the {\bf ground state energy}
\begin{equation}\label{egs}
{\mathrm E}_{\rm gs}(\kappa,H):= \inf \{  \mathcal{E}(\psi,{\bf A})~:~ (\psi,\Ab)\in H^1(\Omega;\Cb)\times H^1(\Omega;\R^2)                          \}.
\end{equation}
As the intensity of the magnetic field  varies (i.e. the parameter $H$), changes in ${\mathrm E}_{\rm gs}(\kappa,H)$ mark various distinct states of the superconductor. That has been fairly understood for type~II superconductors  in the case where the magnetic field is uniform (i.e. $B_0=1$) which has allowed to distinguish between three critical values for the intensity of the applied magnetic field, denoted by $H_{C_1}$, $H_{C_2}$ and $H_{C_3}$ whose role can be described as follows (see \cite{FH-b, SS-b, CR1,CR2,CR3, FK-am}):
\begin{itemize}
\item If $H<H_{C_1}$, then the whole superconductor is in the perfect superconducting state\,;
\item If $H_{C_1}<H<H_{C_2}$, the superconductor is in the mixed phase, where both the superconducting and normal states co-exist in the bulk of the sample; the most interesting aspect of the mixed phase is that the region with the normal state  appears in the form of a lattice of point defects, covering the whole bulk of the sample \cite{SS-cmp}\,;
\item If $H_{C_2}<H<H_{C_3}$, superconductivity disappears  in the bulk but survives on the surface of the superconductor\,;
\item If $H>H_{C_3}$, superconductivity is destroyed and the superconductor returns to the normal state\,.
\end{itemize} 

The case of a non-uniform sign changing magnetic field has been addressed first in \cite{P.K.} then  recently in \cite{Att1, Att2, Att3, HK, HK1}. In the presence of such magnetic fields, the behavior of the superconductor (and the  associated critical magnetic fields) differ significantly from the case of a uniform applied magnetic field. In particular, the  order of  the intensity of the third critical field $H_{C_3}$ increases, and in the mixed phase between $H_{C_2}$ and $H_{C_3}$, superconductivity is  neither present everywhere in the bulk, nor it is evenly distributed in the form of a lattice. We refer to \cite{HK, HK1} for more details.

Now we state our assumption on the function $B_0$. These are two conditions that will allow $B_0$ to represent  a non-uniform sign changing applied magnetic field. The first condition is on the zero set of $B_0$ and says
\begin{equation}\label{eq:Gamma*}
\Gamma:=\{x\in \overline{\Omega}\,,\,B_0(x)=0\}\not=\emptyset
\quad{\rm and}\quad \Gamma\cap\partial\Omega~{\rm is~ finite}\,.
\end{equation}
The second condition is on the gradient of the function $B_0$ and yields that the function $B_0$ vanishes non-degenerately and changes sign:
\begin{equation}\label{Eq:Gamma}
|B_0|+|\nabla B_0|\neq 0 \quad {\rm in} \qquad  \overline\Omega.
\end{equation}
Note that \eqref{Eq:Gamma} yields that $\Gamma$ consists of a finite number of smooth curves that are assumed to  intersect $\partial\Omega$ transversely. Such magnetic fields arise naturally in many contexts \cite{AHP, CL, Mon}.

Under the assumptions \eqref{eq:Gamma*} and \eqref{Eq:Gamma}, the ground state energy $\mathrm E_{\rm gs}(\kappa,H)$ is estimated for various regimes of $H$ and $\kappa$.  Firstly, in light of results in Pan-Kwek \cite{P.K.} and Attar \cite{Att3}, we know that there exists $\overline{M}>0$ such that, for $H>\overline{M}\kappa^2$ and $\kappa$ sufficiently large, ${\rm E_{\rm gs}}(\kappa,H)=0$ and  every critical point $(\psi,\Ab)$ of the functional in \eqref{GL-Energy} is a normal solution, i.e. $\psi=0$ everywhere. The meaning of this is that the critical field $H_{C_3}$, the threshold above which superconductivity is lost, is of the order of $\kappa^2$.
 
In the recent paper \cite{HK1}, the authors write an  asymptotic expansion for  the ground state energy  in the specific regime where $H$ is of order $\kappa^2$ and  $\kappa\to+\infty$ (in this case,  $H$ is of the order of the third critical field $H_{C_3}$). 

The result in \cite{HK1} reads as follows. There exists a universal function $E(\cdot)$, introduced in Theorem \ref{Thm:EL} below, such that if $0<M_1<M_2$, then, for $H\in[M_1\kappa^2, M_2\kappa^2]$, the ground state energy satisfies, as $\kappa\rightarrow\infty$,
\begin{equation}\label{eq:ae-HK}
{\mathrm E}_{\rm gs}(\kappa,H)=\kappa  \int_{\Gamma}\left(|\nabla B_0(x)|\frac{H}{\kappa^2}\right)^{1/3}E\left(|\nabla B_0(x)|\frac{H}{\kappa^2}\right)      ds(x)+\frac{\kappa^3}{H} o (1)\,,
\end{equation}
where $ds$ denotes the arc-length measure in $\Gamma$.

The asymptotic analysis of $\mathrm E_{\rm gs}(\kappa,H)$ has been carried for other regimes of the magnetic field strength, down to $H\approx \kappa^{1/3}$, in \cite{Att1, Att2, HK1}.  The case where the function $B_0$  is only  H\"older continuous or a step function has been discussed in \cite{HK, AK}. 

Let us mention a few properties of the function $E(\cdot)$ appearing in \eqref{eq:ae-HK}:
\begin{itemize}
\item $L\in (0,\infty)\mapsto E(L)\in(-\infty,0]$ is a continuous function\,;
\item As $L\to0_+$, the asymptotic behavior of $E(L)$ is analyzed in \cite{HK2}; in particular $|E(L)|\approx L^{-4/3}$\,;
\item There exists a universal (spectral) constant $\lambda_0>0$  (defined below in \eqref{eq:lambda0}) such that $E(L)=0$ for $L\geq\lambda_0^{-3/2}$ and $E(L)<0$ for $0<L<\lambda_0^{-3/2}$.
\end{itemize}

The aim of this paper is to analyze the asymptotic behavior of $E(L)$ as $L\to\lambda_0^{-3/2}$ from below (thereby complementing the result in \cite{HK2} devoted for the regime $L\to0_+$).  To that end, we introduce the following quantities :
\begin{itemize}
\item $\lambda_0>0$ and $\tau_0<0$ are the constants (see Theorem~\ref{thm:M-op})
\begin{equation}\label{eq:lambda0}
\lambda_0=\inf_{\alpha\in\R}\lambda(\alpha)=\lambda(\tau_0)
\end{equation}
where $\lambda(\alpha)$ is the lowest eigenvalue of the operator $-\frac{d^2}{dt^2}+\left(\frac{t^2}2+\alpha\right)^2$.
\item $u_{0}$ is the positive $L^2$-normalized eigenfunction satisfying
$$\left(-\frac{d^2}{dt^2}+\left(\frac{t^2}2+\tau_0\right)^2\right)u_{0}=\lambda_0u_{0}~{\rm in~}\R\,.$$
\end{itemize}

We obtain:
\begin{thm}\label{main-theorem} As $L\nearrow\lambda_0^{-3/2}$, the  following asymptotic formula holds,
\[
E(L)= -\frac{ L^{2/3}}{2} \dfrac{(L^{-2/3}-\lambda_0)^2 }{\|u_{0}\|_4^4}(1+o(1))\,.
\]
\end{thm}

Now we return back to \eqref{eq:ae-HK} and observe that, when $H$ satisfies
$$\left(\min_{x\in\overline{\Omega}}|\nabla B_0(x)|\right)\frac{H}{\kappa^2}\geq \lambda_0^{-3/2}\,,$$
the leading order term in \eqref{eq:ae-HK} vanishes (so superconductivity disappears in the bulk of the sample). This leads us to introduce the following critical field
\begin{equation}\label{eq:HC2}
H_{C_2}(\kappa)=\gamma\kappa^2
\end{equation}
where 
\begin{equation}\label{eq:gamma}
\gamma:=\lambda_0^{-3/2} c_0^{-1}\quad{\rm and}\quad  c_0=\min_{x\in\Gamma}|\nabla B_0(x)|\,.
\end{equation}
Then one may ask whether we can refine the formula in \eqref{eq:ae-HK}  under the assumption that  
$H$ is close to and below $H_{C_2}(\kappa)$ (see \eqref{condition-on-H} below).
Indeed this is possible by using  Theorem~\ref{main-theorem} and by working under a rather generic assumption on $B_0$:
\begin{assumption}\label{ass:B0-2}
Suppose that $B_0$ satisfies \eqref{eq:Gamma*} and \eqref{Eq:Gamma}.  Let  $c_0$ be the constant introduced in \eqref{eq:gamma} and
\begin{equation}\label{eq:Gamm0}
\Gamma_0=\{x\in\Gamma~:~|\nabla B_0(x)|=c_0\}
\end{equation}
be the set of minimum  points of the function $\Gamma\ni x\mapsto |\nabla B_0(x)|$.

 We assume that {\bf one}  of the following two conditions hold:
\begin{itemize}
\item Either $\Gamma_0=\Gamma$,\,
\item or the set $\Gamma_0$ is finite, $\Gamma_0\subset\Omega$ and  every point of $\Gamma_0$ is a non-degenerate minimum of the function $\Gamma\ni x\mapsto |\nabla B_0(x)|$.
\end{itemize}
\end{assumption}
\begin{rem} In the case of the unit disc $\Omega=B(0,1)$, the following two functions 
$$(x,y)\mapsto y-x\quad{\rm and}\quad (x,y)\mapsto y-x^2$$
serve 
as two examples of a magnetic field $B_0$ satisfying  Assumption~\ref{ass:B0-2}.
\end{rem}

\begin{rem}
If the set $\Gamma_0$ is finite and there exists $x_0\in\Gamma_0\cap\partial\Omega$, then $x_0$ is a non-degenerate minimum  if the derivative of the map $x\mapsto\nabla B_0(x)$ at $x_0$ is not zero. 
\end{rem}

Assumption~\ref{ass:B0-2} is reminiscent  of the assumption by Fournais-Helffer in \cite[Assumption~5.1]{FH-cvpde} but with the function $x\mapsto \big(-|\nabla B_0(x)|\big)$ here replacing  the curvature there. Also, Assumption~\ref{ass:B0-2} appears in the analysis of magnetic mini-wells 
by Helffer-Kordyukov-Raymond-V\~{u}\,Ng\c{o}c \cite{HKRV}.

Next we assume that $H$ approaches the critical field in \eqref{eq:HC2} as follows
\begin{equation}\label{condition-on-H}
H=\Big(\gamma-\rho(\kappa)\Big)\kappa^2\,,
\end{equation}
where  the constant $\gamma$ is introduced in \eqref{eq:gamma} and
\begin{equation}\label{eq:rho}
\rho:(0,\infty)\to(0,\infty) {\rm ~satisfies~} \kappa^{-1/30}\ll  \rho(\kappa)\ll 1\,.
\end{equation}
Here and in the sequel, we use the following notation. If $a(\kappa)$ and $b(\kappa)$ are two positive valued functions, the notation $a(\kappa)\ll b(\kappa)$ means that $a(\kappa)/b(\kappa)\to0$ as $\kappa\to\infty$. Also, by writing $a(\kappa)\approx b(\kappa)$ it is meant that there exist constants $\kappa_0,c_1,c_2>0$ such that $c_1 b(\kappa)\leq a(\kappa)\leq c_2 b(\kappa)$, for all $\kappa\geq\kappa_0$.

Clearly, when \eqref{condition-on-H}, \eqref{eq:rho}  and Assumption~\ref{ass:B0-2} hold, the principal term in \eqref{eq:ae-HK} satisfies
\begin{align}\label{eq:ae-HK*}
&\int_{\Gamma}\left(|\nabla B_0(x)|\frac{H}{\kappa^2}\right)^{1/3}E\left(|\nabla B_0(x)|\frac{H}{\kappa^2}\right)      ds(x)\nonumber\\
&=-\frac{\gamma}{2\|u_{0}\|_4^4}\left( \int_{\Gamma}
 |\nabla B_0(x)|\left(\Big(\frac{H}{\kappa^2}|\nabla B_0(x)|\Big)^{-2/3}-\lambda_0\right)_+^2\,ds(x)\right)(1+o(1))\nonumber\\
&=-\frac{\lambda_0^{-3/2}}{2\|u_{0}\|_4^4}\left( \int_{\Gamma}
 \left(\Big(\frac{H}{\kappa^2}|\nabla B_0(x)|\Big)^{-2/3}-\lambda_0\right)_+^2\,ds(x)\right)(1+o(1))\,.
\end{align}
The last step follows since $\gamma=\lambda_0^{-3/2}c_0^{-1}$ and  the function on $\Gamma$, $\left(\Big(\frac{H}{\kappa^2}|\nabla B_0(x)|\Big)^{-2/3}-\lambda_0\right)_+$, is supported in $\overline{\Gamma_\kappa}$, where
\begin{equation}\label{eq:Gam0}
\Gamma_\kappa=\{x\in\Gamma~:~\frac{H}{\kappa^2}|\nabla B_0(x)|<\lambda_0^{-3/2}\}\,,
\end{equation}
which yields that $|\nabla B_0(x)|\sim c_0$ on $\Gamma_\kappa$.

Under Assumption~\ref{ass:B0-2}, only one of the following two cases may occur:
\begin{itemize}
\item Either $\Gamma_\kappa=\Gamma$, in which case
$$\int_{\Gamma}
\left( \Big(\frac{H}{\kappa^2}|\nabla B_0(x)|\Big)^{-2/3}-\lambda_0\right)_+^2\,ds(x)=\left(\frac23c_0\lambda_0^{3/2}\rho(\kappa)\right)^2|\Gamma|\big(1+o(1)\big)\,;$$
\item or  $|\Gamma_\kappa|\approx \sqrt{\rho(\kappa)}$ as $\kappa\to+\infty$, in which case 
$$\int_{\Gamma}
 \left(\Big(\frac{H}{\kappa^2}|\nabla B_0(x)|\Big)^{-2/3}-\lambda_0\right)_+^2\,ds(x)\geq c\rho(\kappa)^2\sqrt{\rho(\kappa)}\approx \big(\rho(\kappa)\big)^{5/2}\,,$$
 for some constant $c>0$, which depends on the second derivative of the function $|\nabla B_0(x)|$ at the minimum points.
 \end{itemize}

 As an application of the main result of this paper (Theorem~\ref{thm:M-op}), we are able to prove that


\begin{thm}\label{corol:KN1}
Under Assumption~\ref{ass:B0-2}, if \eqref{condition-on-H} and \eqref{eq:rho} hold, then as $\kappa\to+\infty$,
$$
{\mathrm E}_{\rm gs}(\kappa,H)=\left[ -\frac{\kappa\lambda_0^{-3/2}}{2\|u_{0}\|_4^4} \int_{\Gamma}
 \left(\Big(\frac{H}{\kappa^2}|\nabla B_0(x)|\Big)^{-2/3}-\lambda_0\right)_+^2\,ds(x)\right]\big(1 +o(1)\big)\,.
$$
\end{thm}
The result in Theorem~\ref{corol:KN1} is far from optimal. We mention it as a simple application of Theorem~\ref{main-theorem} and the analysis in \cite{HK1}. To get the optimal regime (for $\rho(\kappa)$) where the result in Theorem~\ref{corol:KN1} holds, we need a rather detailed analysis of the ground state energy and the corresponding minimizers, that we postpone to a separate work. 

The rest of the paper is organized as follows. We introduce in Section \ref{S-GL} a certain simplified Ginzburg-Landau functional from which arises the definition of the limiting function $E(L)$ appearing in Theorem~\ref{main-theorem} above. We recall in Section \ref{Montgomery} spectral facts concerning the family of Montgomery operators. A related family of 1D linear functionals is introduced in Section \ref{non-linear-funct} where we investigate the infimum over all the ground state energies of those functionals. Moreover, we prove in Section \ref{non-linear-funct} a key-ingredient asymptotic formula needed for the proof of the main result. A technical spectral estimate is proved in  Section \ref{Spec-est}. We perform in Section \ref{models-HC} some Fourier analysis to get a good estimate on the energy functional defined on half-cylinders. We conclude with the proof of Theorem \ref{main-theorem} in Section~\ref{Proof-MT}. Finally, in Section~\ref{sec:proof-KN1}, we prove Theorem~\ref{corol:KN1}.

\section{The simplified Ginzburg-Landau functional}\label{S-GL}
We consider the following  magnetic potential, 
\begin{equation}\label{Aapp}
{\bf A}_{\rm app}(x)=\left(-\dfrac{x_{2}^{2}}{2},0\right)\quad\big(x=(x_1,x_2)\in\R^2\big)
\end{equation}
which generates the magnetic field $\curl\Ab_{\rm app}=x_2$ that vanishes along the line $x_2=0$.

Let $L>0,b>0,R>0$ and $S_{R}=(-R,R)\times \R$. Consider the functional 
\begin{equation}{\label{E:Rb}}
\mathcal{E}_{R,b}(u)=\int_{S_{R}}\left( |(\nabla -i{\bf A}_{\rm app})u|^{2}-b|u|^{2}+\dfrac{b}{2}|u|^{4}\right) dx,
\end{equation}
and the corresponding ground state energy 
\begin{equation}\label{ebr}
\mathfrak{e}(b;R)= {\rm inf}\left\{ \mathcal{E}_{R,b}(u)  ~:~   (\nabla -i{\bf A}_{\rm app})u\in L^{2}(S_{R}), ~ u\in L^{2}(S_{R})~ {\rm and}~ u=0 ~{\rm on}~\partial S_{R}  \right\}\,.
\end{equation}
The following theorem was proven in  \cite[Theorem 3.8]{HK1}. 
\begin{thm}\label{Thm:EL}
Given $L>0$, there exists $E(L)\leq 0$ such that,
\begin{equation}\label{eq:E(L)}
\lim_{R\rightarrow\infty}\dfrac{\mathfrak{e}(L^{-2/3};R)}{2R}=E(L).
\end{equation}
The function $(0,\infty)\ni L\mapsto E(L)\in (-\infty,0] $ is continuous, monotone increasing, and 
\[
E(L)=0\quad \mbox{ if and only if } \quad L\geq \lambda_{0}^{-3/2}\,,
\]
where $\lambda_0>0$ is the eigenvalue introduced in \eqref{eq:lambda0}.
 
Furthermore, there exists a constant $C>0$ such that
\begin{equation}
\forall R\geq 2, ~\forall L>0, \qquad  E(L)\leq \dfrac{e(L^{-2/3};R)}{2R}\leq E(L)+ C(1+L^{-2/3})R^{-2/3}.
\end{equation}
\end{thm}

\section{The Montgomery operator}\label{Montgomery}

For $\alpha\in\R$, consider the self-adjoint operator in $L^{2}(\R)$,  
\begin{equation}\label{eq:P-alpha}
P(\alpha) =-\dfrac{d^{2}}{dt^{2}}+\left(\dfrac{t^{2}}{2} +\alpha   \right)^{2}
\end{equation}
with domain 
\begin{equation}\label{B2R}
{\rm Dom}\big(P(\alpha)\big)=B^2(\R)=\Big\{ f\in H^{2}(\R)~:~ t^{4}f \in L^{2}(\R)\Big\}.
\end{equation} 
The first eigenvalue $\lambda(\alpha)$ of the operator $P(\alpha)$ is expressed by the min-max principle as follows
\begin{equation}\label{Eq:def-lambda-alpha}
\lambda(\alpha):=\inf_{u\in B^1({\R})} \dfrac{Q_\alpha(u)}{\|u\|^2_2},
\end{equation}    
where
\begin{equation}\label{Qalpha}
Q_{\alpha}(u)=\int_{\R} \left(   |u^{\prime}(t)|^{2}+ \left(\dfrac{t^{2}}{2}+\alpha \right)^{2}|u(t)|^{2}\right)dt
\end{equation}
is the quadratic form defined for $u$ in the space
\begin{equation}\label{B2R*}
B^1(\R)=\{u\in H^1(\R)~:~t^2u\in L^2(\R)\}\,.
\end{equation}
Recall that $\lambda_0=\inf_{\alpha\in\R}\lambda(\alpha)$ introduced in \eqref{eq:lambda0}. We collect from \cite{B.H.notes} some important properties of the function $\alpha\mapsto\lambda(\alpha)$.
\begin{thm}\label{thm:M-op}~
\begin{enumerate}

\item There exists a unique $\tau_{0}\in\R$ such that 
\(
\lambda_{0}=\lambda(\tau_{0}).
\)
\item $\tau_{0}<0$ and $\lambda_{0}<\lambda(0)\leq \left( \dfrac{3}{4}\right)^{\frac{4}{3}}<1$.
\item $\displaystyle\lim_{\alpha\rightarrow\pm \infty}\lambda(\alpha)=+\infty.$
\item The minimum of $\lambda$ at $\tau_{0}$ is non-degenerate, that is, $\lambda^{\prime\prime}(\tau_{0})>0$.
\end{enumerate}

\end{thm}

\begin{rem}\label{rem:lambda0}
One finds the  numerical approximation $\lambda _{0}\cong 0.57$ (see. \cite{H.M., Mon}).
\end{rem}

As a consequence of Theorem~\ref{thm:M-op}, we may define two functions $z_1(b)$, $z_2(b)$ satisfying 
\[
z_1(b)< \tau_0< z_2(b), \qquad \lambda^{-1}([\lambda_0,b))=(z_1(b), z_2(b)).
\]

For all $\alpha\in\R$,  let $\lambda_2(\alpha)$ be the second eigenvalue of the operator  $P(\alpha)$ introduced in \eqref{eq:P-alpha}. By  continuity of  the functions $\alpha\mapsto\lambda_n(\alpha)$, for all $n\in\{1,2\}$, we get 
\begin{lem}\label{sec-eig}
Let $\tau_0$ be the value defined in Theorem~\ref{thm:M-op}. There exists $\varepsilon_0>0$ such that, if $\alpha\in(\tau_0-\varepsilon_0,\tau_0+\varepsilon_0)$ and $b\in[\lambda_0,\lambda_0+\varepsilon_0)$, then $b<\lambda_2(\alpha)$.
\end{lem}

In the sequel, we consider $\alpha\in(\tau_0-\varepsilon_0,\tau_0+\varepsilon_0)$ and $b\in[\lambda_0,\lambda_0+\varepsilon_0)$, where $\varepsilon_0$ is defined by Lemma~\ref{sec-eig}\,.
Let $u_\alpha$ be the positive normalized ground state of the operator $P(\alpha)$, and let $\pi_\alpha $ be the $L^2$ orthogonal projection on ${\rm Span}(u_\alpha)$. For $\alpha=\tau_0$, we shorten the notation and write $u_0:=u_{\tau_0}$. 

We introduce the regularized resolvent of $P(\alpha)$ by
\begin{equation}\label{eq:reg-resolvent}
R_{\alpha,b}:=(P(\alpha)-b)^{-1}(1-\pi_{\alpha})\,.
\end{equation}

The following lemma is straightforward (see \cite[Lem.~14.2.6]{FH-b}):

\begin{lem} \label{Continuity-Res}
The regularized resolvent $R_{\alpha,b}$ maps $L^2({\R})$ into $B^2(\R)$.
 Moreover, there exist $\varepsilon,C>0$  such that for all $(\alpha,b)\in (-\tau_0-\varepsilon,\tau_0+\varepsilon)\times[\lambda_0,\lambda_0+\varepsilon)$,
\[
\| R_{\alpha, b}u\|_{B^2(\R)}\leq C \|u\|_{L^2(\R)}.
\]
\end{lem}

\section{A family of $1D$ non-linear functionals}\label{non-linear-funct}

Let $b>0$ and $\alpha \in \R$. Consider the functional
\begin{equation}\label{GLF-M*}
B^1(\R)\ni f\mapsto\mathcal{E}_{\alpha,b}(f)=\int_{\R} \left(   |f^{\prime}(t)|^{2}+ \left(\dfrac{t^{2}}{2}+\alpha \right)^{2}|f(t)|^{2}-b|f(t)|^{2}+\frac{b}{2}|f(t)|^{4}\right)dt\,,
\end{equation}
along with the ground state energy
\begin{equation}\label{GLF-M:gse}
\mathfrak{b}(\alpha,b)=\inf\{\mathcal{E}_{\alpha,b}(f)~:~f\in B^1(\R) \}\,,
\end{equation} 
where $B^1(\R)$ is the space introduced in \eqref{B2R*}. We continue to work under the assumptions made in Theorem~\ref{thm:M-op} and afterwards.

Our objective is to prove 

\begin{thm} \label{prop-min}
There exists $b_*>0$ such that, if  $\lambda_0<b\leq b_*$, then there exists a unique $\xi(b)\in (z_1(b),z_2(b))$ satisfying
\[
\mathfrak{b}(\xi(b),b)=\inf_{\alpha\in\R} {\mathfrak b}(\alpha,b)\,.
\]
Furthermore, 
\begin{itemize}
\item the function $b\mapsto \xi(b)$ is a $C^\infty$ function on $(\lambda_0,b_*]$ with $\xi(\lambda_0)=\tau_0$\,;
\item As $b\searrow \lambda_0 $, $\displaystyle
 \inf_{\alpha} \mathfrak{b}(\alpha,b)=-\frac{1}{2b} \dfrac{(b-\lambda_0)^2 }{\|u_{0}\|_4^4}\big(1+o(1)\big)$\,.
\end{itemize}
\end{thm}

The starting point is the following preliminary result:

\begin{thm}\label{GLF-M}
Let $b>0$ and $\alpha\in \R$.  Then the following hold:
\begin{enumerate}
\item The functional $\mathcal{E}_{\alpha,b}$ has a strictly positive minimizer $f_{\alpha,b}$ in the space $B^1(\R)$ if and only if $\lambda(\alpha)<b$ . Furthermore, the minimizer satisfies the Euler-Lagrange equation
\begin{equation}\label{eqt-f}
-f^{\prime\prime}_{\alpha,b}+\left( \frac{t^2}{2}+\alpha                    \right)^2 f_{\alpha,b}=bf_{\alpha,b} (1-|f_{\alpha,b}|^2),
\end{equation}
and the inequality 
\[
\|f_{\alpha,b}\|_{\infty}\leq 1.
\]

\item The ground state energy in \eqref{GLF-M:gse} satisfies
\begin{equation}\label{inf-formula}
\mathfrak{b}(\alpha,b)= -\frac{b}{2}\|f_{\alpha,b}\|_4^4.
\end{equation}
\item
There exists $\alpha_0\in (z_1(b), z_2(b))$ such that, 
\[
\mathfrak{b}(\alpha_0,b)=\inf_{\alpha \in \R} \mathfrak{b}(\alpha,b).
\]
\item If $b< \lambda(0)$, then $\alpha_0<0.$
\item If $b >0$. The map 
$
\alpha\in(z_1(b), z_2(b))\mapsto f_{\alpha,b}\in B^1(\R)  
$
is $C^{\infty}.$
\item (Feynman-Hellmann) 
\begin{equation}\label{Fey-Hel}
\int_{\R}\left(\frac{t^2}{2}+\alpha_0\right)|f_{\alpha_0}(t)|^2 dt=0.
\end{equation}
\end{enumerate}
\end{thm}

The proof of Theorem~\ref{GLF-M} is obtained by adapting the same analysis of \cite[Section 14.2]{FH-b} devoted to the functional 

\begin{equation}\label{eq:GLF-Pan}
\mathcal{F}_{\alpha,b}(f)=\int_{\R} \left(   |f^{\prime}(t)|^{2}+ \left(t+\alpha \right)^{2}|f(t)|^{2}-b|f(t)|^{2}+\frac{b}{2}|f(t)|^{4}\right)dt.
\end{equation}

\begin{rem}\label{rem:GLF-M*}
The existing results on the functional  in \eqref{eq:GLF-Pan} suggest that Theorem~\ref{prop-min} holds for all $b>\lambda_0$ (see \cite{CR1,CR2,CR3}). However, in the new functional \eqref{GLF-M*}, the presence of the non-translation invariant  potential term 
$\left(\frac{t^2}2+\alpha\right)$
causes technical difficulties that prevent the application of the  method of \cite{CR1,CR2,CR3}. 
\end{rem}

According to Theorem~\ref{GLF-M}, we observe that the functional $\mathcal{E}_{\alpha,b}$ has non-trivial minimizers if and only if $\alpha \in (z_1(b), z_2(b))$. Furthermore, as $b\searrow \lambda_0$, $z_1(b),z_2(b)\to\tau_0$ and consequently, $\alpha_0\to\tau_0$. So, if $b$ is sufficiently close to $\lambda_0$, the minimum points of the function $\alpha\mapsto \mathfrak b(\alpha,b)$ are localized in a neighborhood of $\tau_0$.  

In the sequel, we  assume that the pair $(\alpha,b)$ lives in a sufficiently small neighborhood of $(\tau_0,\lambda_0)$ so that the results in Section~\ref{Montgomery} hold.

\begin{lem}\label{lem:delta*}
Let   
\begin{equation}\label{delta}
\delta=\langle  f_{\alpha,b}, u_{\alpha}    \rangle\,.
\end{equation}
Then
\begin{equation}\label{Res1}
(b-\lambda(\alpha))\delta=b \langle f_{\alpha,b}^3, u_{\alpha}\rangle,
\end{equation}
and
\begin{equation}\label{Res2}
f_{\alpha,b}+b R_{\alpha, b}(f^3_{\alpha,b})=\delta u_{\alpha}\,.
\end{equation}
\end{lem}
\begin{proof}
The formula in \eqref{Res1} results from \eqref{eqt-f} because $P(\alpha)u_\alpha=\lambda(\alpha) u_\alpha$.
Next we prove \eqref{Res2}. 
Note that $\pi_\alpha(P(\alpha)-b)f_{\alpha,b}=\delta (P(\alpha)-b)u_\alpha$. We may write \eqref{eqt-f} as $bf_{\alpha,b}^3=-\big(P(\alpha)-b\big)f_{\alpha,b}$. Consequently,
\begin{align*}
 R_{\alpha, b}(bf^3_{\alpha,b})&= (P(\alpha)-b)^{-1}(bf^3_{\alpha,b}-\pi_{\alpha}(bf^3_{\alpha,b}))\\
 &= - f_{\alpha,b}+\delta u_{\alpha}.
\end{align*}
Here is the identity in \eqref{Res2}.
\end{proof}

Since $B^1(\R)$ is embedded in $L^\infty(\R)$, we  can define the following map
\begin{equation}\label{eq:G}
B^1(\R)\ni u\mapsto G_{\alpha,b}(u)=-b R_{\alpha,b} (u^3)\,.
\end{equation}

As a consequence of Lemma \ref{Continuity-Res}, we find
\begin{lem}\label{lem:G}
There exist a neighborhood $\mathcal N_0=(\tau_0-\varepsilon,\tau_0+\varepsilon)\times[\lambda_0,\lambda_0+\varepsilon)$ and  a constant $C>0$ such that, for all $(\alpha,b)\in\mathcal N_0$, the map $G_{\alpha,b}$ maps $B^1(\R)$ to itself, and   for all $u\in B^1(\R)$,
\[
\|G_{\alpha,b}(u)\|_{B^1(\R)}\leq C \|u\|^3_{B^1(\R)}\,.         
\]
\end{lem}

With Lemma~\ref{lem:G} in hand, we can  invert equation $(I-G_{\alpha,\beta})(u)=f$ when the pair $(\alpha,b)$ lives in the neighborhood $\mathcal N_0$, and the norm of $u$ is sufficiently small. We state this as follows.

\begin{lem}\label{corol:G}
There exists a constant $c_*>0$ such that, for all  $(\alpha,b)\in\mathcal N_0$ and $u\in B^1(\R)$ satisfying  $\|u\|_{B^1(\R)}\leq c_*$, the series
$$\mathfrak{t}(u)=  \sum_{j=0}^{\infty} G_{\alpha,b}^j(u)$$
is absolutely convergent.
 Furthermore,
$$\mathfrak t\Big(u-G_{\alpha,b}(u)\Big)=u\,.$$
\end{lem}

Now we return back to \eqref{Res2} and observe that it can be expressed in the following form 
\begin{equation}\label{eq:G*}
(I-G_{\alpha,b})(f_{\alpha,b})=\delta u_{\alpha}.
\end{equation}
We will apply  Lemma~\ref{corol:G} to invert the  formula \eqref{eq:G*}, but we have to prove first that $\|f_{\alpha,b}\|_{B^1(\R)}$ is sufficiently small, which is our next task.

\begin{lem}\label{lem:norm-f}
There exists a constant $C>0$ such that, for all $\alpha\in\R$ and $b\geq\lambda_0$,  we have
\[
\|f_{\alpha,b}\|_{B^1(\R)}\leq  C b^{3/2} \sqrt{b-\lambda_0}.
\]
\end{lem}
\begin{proof}
We can find a constant $C_1>0$ such that, for  all $f\in B^1(\R)$ and $(\alpha,b)$,  
\begin{equation}\label{eq:revision}
\begin{aligned}
\|f\|^2_{B^1(\R)}& \leq  C_1\int_{\R}  \left[   |f^\prime(t)|^2 +\left(\alpha +\frac{t^2}2\right)^2 |f(t)|^2                                                       \right]dt\\
&=C_{1}\left\{\mathcal E_{\alpha,b}(f)+\int_\R \left(b|f|^2-\frac{b}2|f|^4\right)\,dx\,\right\},
\end{aligned}
\end{equation}
where $\mathcal E_{\alpha,b}(\cdot)$ is the functional introduced in \eqref{GLF-M*}.

Now we choose  $f=f_{\alpha,b}$. Consequently $\mathcal{E}_{\alpha,b}(f)\leq\mathcal E_{\alpha,b}(0)=  0$. So we can drop the term $\mathcal E_{\alpha,b}(f)$ from \eqref{eq:revision} and get the following two inequalities, 
\begin{equation}\label{NB1-eq1}
\|f\|^2_{B^1(\R)} \leq C_1 b \|f\|^2_{2},
\end{equation}
and
\begin{equation}\label{NB1-eq2}
\begin{aligned}
\frac{b}{2}\|f\|^4_4 &\leq b\|f\|_2^2- \int_{\R}  \left[   |f^\prime(t)|^2 +\left(\alpha +\frac{t^2}2\right)^2 |f(t)|^2                                                       \right]dt\\
&\leq (b-\lambda(\alpha))\|f\|_2^2 \leq (b-\lambda_0)\|f\|^2_{B^1(\R)}~{\rm ~since~}\lambda(\alpha)\geq\lambda_0.
\end{aligned}
\end{equation}
On the other hand, using H\"older's inequality, we write
\begin{multline}\label{NB1-eq3}
\|f\|_{2}^2 =\int_{\R} |f(t)| (1+t^2) |f(t) |(1+t^2)^{-1} dt\\
\leq      \|f\|_4 \|(1+t^2)f\|_2\|(1+t^2)^{-1}\|_4\leq C_2    \|f\|_4 \|f\|_{B^1(\R)}\,,
\end{multline}
for some constant $C_2$ independent of $(\alpha,b)$. Combining \eqref{NB1-eq1}-\eqref{NB1-eq3} gives, for $C_3=2^{1/4}C_1C_2$
\[
\|f\|^2_{B^1(\R)}\leq  C_3  b^{3/4}(b-\lambda_0)^{1/4}\,\|f\|^{3/2}_{B^1(\R)}.
\]
This yields the conclusion in Lemma~\ref{lem:norm-f} with  $C=C_3^2$.
\end{proof}

In the sequel, {\bf we assume the additional condition $C b^{3/2} \sqrt{b-\lambda_0}<c_*$,} where $c_*$ is the constant in Lemma~\ref{corol:G}. Now, Lemma~\ref{lem:norm-f} and the identity \eqref{eq:G*}  yield:

\begin{lem}\label{B1norm-f}
There exists $\varepsilon>0$ such that, for all $(\alpha,b)\in(\tau_0-\varepsilon,\tau_0+\varepsilon)\times [\lambda_0,\lambda_0+\varepsilon)$,  the function $f_{\alpha,b}$ satisfies, 
\begin{equation}\label{series}
f_{\alpha,b}
=\sum_{j=0}^{\infty} \delta^{3^j}G_{\alpha,b}^{j}( u_{\alpha})\,,
\end{equation}
where $\delta$ is introduced in \eqref{delta}.
\end{lem}
~

\begin{proof}[Proof of Theorem~\ref{prop-min}]~

{\bf Step~1: A spectral expression for $f_{\alpha,b}$.} 

The definition of $\delta$ in \eqref{delta} and Lemma~\ref{lem:norm-f} yield 
\[
0\leq \delta\leq C b^{3/2}\sqrt{b-\lambda_0}\,.
\]
Assuming $b-\lambda_0$ is sufficiently small, we get $0\leq \delta<1$. Consequently, the series
\begin{equation}\label{eq:S}
S(\delta,\alpha,b):=\sum_{j=0}^{\infty} \delta^{3^j}G_{\alpha,b}^{j}( u_{\alpha}),
\end{equation}
is normally convergent in the space $B^1(\R)$ and depends smoothly on the parameters $(\delta,\alpha,b)$.

Later, it will be convenient to write
\[
S(\delta,\alpha , b )= \delta T(\delta^2,\alpha,b),
\]
where, for $\epsilon>0$,
\begin{equation}\label{series-f-2}
T(\epsilon,\alpha,b)= \sum_{j=0}^{\infty} \epsilon^{\frac{3^j-1}{2}} G_{\alpha,b}^j (u_{\alpha})\,.
\end{equation}
Now Lemma~\ref{B1norm-f}  reads 
\begin{equation}\label{series-f}
f_{\alpha,b}= \delta T(\delta^2,\alpha,b).
\end{equation}
The advantage of \eqref{series-f} is that $f_{\alpha,b}$ is expressed in terms of the spectral quantity $ T(\delta^2,\alpha,b)$ and the value $\delta=\langle f_{\alpha,b},u_\alpha\rangle$. We will use \eqref{series-f} to write a non-trivial relation between the parameters $\alpha,b,\delta$ which will allow us to select the optimal $\alpha$ which minimizes the ground state energy $\mathfrak b(\alpha,b)$ (see \eqref{GLF-M:gse}).  Indeed, there exists a smooth function $n(z_1,z_2,z_3)$ defined in a neighborhood of $(0,\tau_0,\lambda_0)$ such that $n(0,\alpha,b)>0$ for $(\alpha,b)\not=(\tau_0,\lambda_0)$, and (see \cite[Lem.~14.2.9, Eq.~(14.46)]{FH-b})
\begin{equation}\label{eq:delta*}
\delta=\delta(\alpha,b)= \sqrt{b^{-1}(b-\lambda(\alpha))n(b-\lambda(\alpha),\alpha, b)}\,.
\end{equation}
So we can write $f_{\alpha,b}$ in the form (using \eqref{series-f})
\begin{equation}\label{seriesf-1}
f_{\alpha,b}=\delta T\big(\epsilon(\alpha,b) , \alpha,b\big)\,,
\end{equation}
with $\epsilon(\alpha,b)=\delta(\alpha,b)^2$. This proves that $f_{\alpha,b}$ depends smoothly on $(\alpha,b)$ near $(\tau_0,\lambda_0)$.

~

{\bf Step~2: Uniqueness of $\xi(b)$.}

By Theorem~\ref{GLF-M}, we know that a minimum $\alpha_0$ for the function $\alpha\mapsto\mathfrak b(\alpha,b)$ exists, and if $b$ is selected sufficiently close to $\lambda_0$, $\alpha_0$ is localized near $\tau_0$. In this case, it is enough to consider   $\alpha$ varying in a neighborhood of $\tau_0$. In particular, we may assume that \eqref{eq:delta*} holds.

We will prove that any minimum $\alpha_0$, when close enough to $\tau_0$, is unique and depends smoothly on $b$.  
Using \eqref{Fey-Hel} and \eqref{seriesf-1}, we have
\begin{multline}\label{zero-qty}
0=\int_{\R} \left(               \frac{t^2}{2}+\alpha_0\right)|f_{\alpha_0,b}(t)|^2 dt
= \delta(\alpha_0,b) ^2   \int_{\R}      \left(\frac{t^2}{2}+\alpha_0\right)              |T(\epsilon(\alpha_0,b), \alpha_0,b)|^2 dt.\\
 \end{multline}
By the Feyman-Hellman formula for the eigenvalue $\lambda(\alpha)$, we write
\begin{equation}\label{eq:lambda'(alpha)}
\begin{aligned}
\lambda^{\prime}(\alpha_0)&=2 \int_{\R}\left(               \frac{t^2}{2}+\alpha_0\right)       |u_{\alpha_0}(t)|^2= 2\int_{\R}\left(               \frac{t^2}{2}+\alpha_0\right)       |T(0,\alpha_0,b)|^2 dt\\
&=-2\int_{\R}\Big(               \frac{t^2}{2}+\alpha_0\Big)  \Big(    |T(\epsilon (\alpha_0,b), \alpha_0,b)|^2 - |T(0,\alpha_0,b)|^2 \Big)dt,
\end{aligned}
\end{equation}
where we have used \eqref{zero-qty} in the step.

By \eqref{series-f-2}, we see that
\[
\lambda^\prime(\alpha_0)=\delta (\alpha_0,b)^2  a(\alpha_0,b).
\]
where $a(\alpha,b )$ is a smooth function, thanks to \eqref{eq:delta*}.

Using the expression of $\delta (\alpha_0,b)$ in \eqref{eq:delta*}, we see that $\alpha_0$ is a solution of the following equation 
\[
\lambda^\prime(\alpha_0)= (b-\lambda(\alpha_0)) \widetilde{a}(\alpha_0,b)
\]
for a new smooth function $\widetilde{a}$.

Now, the function
\[
(\alpha, b)\mapsto h(\alpha, b):=  \lambda^{\prime}(\alpha)-(b-\lambda(\alpha)) \widetilde{a}(\alpha,b).
\]
satisfies $h(\tau_0,\lambda_0)=0$ since $\lambda'(\tau_0)=0$ and $\lambda(\tau_0)=\lambda_0$. Furthermore,
\[
\dfrac{\partial  h}{\partial \alpha}(\tau_0,\lambda_0)=\lambda^{\prime\prime}(\tau_0)>0.
\]
By the implicit function theorem, there exists a neighborhood $\mathcal N_0$ of $(\tau_0,\lambda_0)$ such that, in this neighborhood,  the equation $h(\alpha,b)=0$ has a unique solution  given by $\alpha=\xi (b)$, where $\xi$ is a smooth function of $b$. 

By selecting $b$  sufficiently close to $\lambda_0$, we get that  $(\alpha_0,b)\in\mathcal N_0$ and satisfies $h(\alpha_0,b)=0$. Consequently, $\alpha_0=\xi(b)$.
~

{\bf Step~3: Asymptotic behavior of the ground state energy.}

We will prove that, as $b\searrow\lambda_0$, 
\begin{equation}\label{eq:asymptot-en}
\|{f_{\xi(b),b}}\|_{4}^4=b^{-2} \dfrac{(b-\lambda_0)^2 }{\|u_{0}\|^4_4}(1+o(1))
\end{equation}
which in turn yields, by Theorem~\ref{GLF-M}, the desired asymptotic expansion for the ground state energy $\mathfrak b(\xi(b),b)$. Recall that, for the ease of the notation, we write $u_0=u_{\tau_0}$.

By the series representation \eqref{series} of $f_{\alpha ,b}$  in the $B^1$-norm (and therefore in the $L^4$-norm) we get
\[
\|{f_{\xi(b),b}}\|_{4}^4      = |\delta|^4 \|u_{\xi(b)}\|_4^4+\mathcal{O}(|\delta|^6).
\]
By smoothness of the function  $b\mapsto \xi(b)$ and $\alpha\mapsto u_\alpha$, we get $\|u_{\xi(b)}\|_4=\|u_{\tau_0}\|_4^4(1+o(1))$, which in turn yields \eqref{eq:asymptot-en}.
\end{proof}

\section{The spectral estimate}\label{Spec-est}
Let $b$ and $\xi(b)$ be as in Theorem~\ref{prop-min}, and let $\beta\in\R$. We introduce $\gamma(\beta,b)$ to be the infimum of the spectrum of the self-adjoint operator associated with the quadratic form 
\begin{equation}\label{Eq:Qbetab}
Q_{\beta, b}(u)= \int_{\R} \left( |u^{\prime}(t)|^2+\left(\frac{t^2}{2}+ \xi(b)+\beta\right)^2 |u(t)|^2-b (1-|f_{\xi(b),b}|^2)|u(t)|^2                                    \right)dt .
\end{equation}
More precisely, using  the min-max principle, 
\begin{equation}\label{def-ld-beta-b}
\gamma(\beta,b):= \inf_{u\in B^1(\R)}\dfrac{ Q_{\beta,b}(u)}{\int_{\R}|u|^2dt}.
\end{equation}
The eigenvalue $\gamma(\beta,b)$ is simple, and by analytic perturbation theory, $\beta\mapsto\gamma(\beta,b)$ is an analytic function. Furthermore, if $u_{\beta,b}$ is a normalized ground state of $\gamma(\beta,b)$, then it depends analytically on $\beta$ as well.

In the sequel, we write
\[
{\gamma}_{ \beta}(\beta,b):=\frac{\partial\gamma}{\partial\beta}(\beta,b)\quad{\rm and}\quad {\gamma}_{ \beta\beta}(\beta,b):=\frac{\partial^2\gamma}{\partial\beta^2}(\beta,b).
\]
Our objective is to prove

\begin{thm}\label{inf-lambda}
There exists $\epsilon>0$ such that for $b\in [\lambda_0, \lambda_0+\epsilon)$, we have 
\[
\inf_{\beta\in\R} \gamma(\beta, b)=0.   
\]
\end{thm}

Theorem~\ref{inf-lambda} has been proved in \cite[Lem~2.2]{Al.He.} for the potential term $(t+\xi)^2$ (instead of $(t^2/2+\xi)^2$ in the expression of $Q_{\alpha,\beta}$). The proof of \cite{Al.He.} can be easily adapted to handle our case where the potential term is $(t^2/2+\xi)^2$. We start by giving  some properties of ${\gamma}(\beta,b)$ when $\beta=0$. 
  
\begin{proposition}\label{pprties-ld-beta-b}
We have:
\begin{enumerate}
\item \label{ld-at-0} $\gamma(0, b)=0$ and ${ \gamma}_{ \beta}(0,b)=  0$, for all $b>\lambda_0$. 
\item \label{second-derivative} 
\begin{equation}
\lim_{b\searrow \lambda_0}\gamma_{\beta\beta}(0,b)=\lambda^{\prime\prime}(\tau_0).
\end{equation}
\end{enumerate}
\end{proposition}
\begin{proof}
Let $u_{\beta,b}$ denote the {\bf unique} positive normalized ground state of $\gamma(\beta,b)$. The function $u_{\beta,b}$ satisfies the eigenvalue equation 
\begin{equation}\label{Eq:ubetab}
-u_{\beta,b}^{\prime\prime}+ \left(\frac{t^2}{2}+ \xi(b)+\beta\right)^2 u_{\beta,b}-b (1-|f_{\xi(b),b}|^2)u_{\beta,b}= \gamma(\beta,b)u_{\beta,b}.
\end{equation}
We set $\beta=0$ and multiply the above equation by $f_{\xi(b),b}$, then we integrate over $\R$ to get
\[
\gamma(0,b)\int_{\R} f_{\xi(b),b}(t) u_{\beta,b}(t)dt=0\,.
\]
Since $f_{\xi(b),b}$ and $u_{\beta,b}$ are positive,  $\displaystyle\int_{\R} f_{\xi(b),b}(t) u_{\beta,b} (t)dt\neq 0$. Thus $\gamma(0,b)=0$ and it follows from \eqref{eqt-f} that
\[
 u_{0,b}=\frac{f_{\xi(b),b}}{\|f_{\xi(b),b}\|_2}.
\]
To prove the statement on the derivative of $\gamma$,  we write the Hellmann-Feynman formula
\[
\dfrac{\partial \gamma}{\partial \beta}(\beta,b)=2\int_{\R}\left(\frac{t^2}{2}+ \xi(b)+\beta\right)|u_{\beta,b}(t)|^2dt.
\]
For $\beta=0$, $u_{\beta,b}=f_{\xi(b),b}/\|f_{\xi(b),b}\|_2$ and  we obtain
\[
\dfrac{\partial \gamma}{\partial \beta}(0,b)=\dfrac{2}{\|f_{\xi(b),b}\|^2_2}\int_{\R}\left(\frac{t^2}{2}+ \xi(b)\right)|f_{\xi(b),b}|^2dt=0~\rm by ~\eqref{Fey-Hel}. 
\]
It remains to prove \eqref{second-derivative}. Note that $z_1(b), z_2(b)\rightarrow \tau_0$ as $b\rightarrow \lambda_0$. Since 
$
z_1(b)< \xi(b)< z_2(b),
$
\begin{equation}\label{xib-tau0}
\xi(b)\rightarrow \tau_0\quad {\rm as}~b\rightarrow \lambda_0\,.
\end{equation}
It follows from Corollary \ref{B1norm-f} that
\[
\lim_{b\rightarrow \lambda_0}{\|f_{\xi(b),b}\|_{B^1(\R)}}=0\,.
\]
By the continuous embedding $B^1(\R)\hookrightarrow L^{\infty}(\R)$, we infer that
\begin{equation}\label{norm-infty}
\lim_{b\rightarrow \lambda_0}{\|f_{\xi(b),b}\|_{L^{\infty}(\R)}}=0\,.
\end{equation}
Note that, for all $u\in B^1(\R)$, 
\[
Q_{\beta,b}(u)= Q_{\xi(b)+\beta}(u)- b \int_{\R}(1-|f_{\xi(b),b}|^2)|u|^2 dt,
\]
where $(\alpha,a) \mapsto Q_{\alpha,a}(\cdot)$ is the quadratic form defined in \eqref{Eq:Qbetab}, and $\alpha\mapsto Q_\alpha(\cdot)$ is the quadratic form introduced  in \eqref{Qalpha}.

Recall the definitions of $\gamma$ and $\lambda$ from \eqref{def-ld-beta-b} and \eqref{Eq:def-lambda-alpha} respectively. Using the min-max principle 
 we get
$$ \lambda( \xi(b)+\beta)-b\leq \gamma(\beta,b)\leq \lambda(\xi(b)+\beta)-b+\|f_{\xi(b),b}\|_\infty^2\,.$$
 It follows from  \eqref{xib-tau0}  and \eqref{norm-infty} that
\begin{equation}
\gamma(\beta,b)\underset{b\rightarrow \lambda_0}{\longrightarrow} \lambda(\tau_0+\beta)-\lambda_0,
\end{equation}
 where the convergence is uniform  (with respect to $\beta$) on every bounded interval in $\R$.

Since $\gamma$ is holomorphic in $\beta$, the derivatives must converge uniformly as well, hence 
\[
\gamma_{\beta\beta}(\beta,b)\rightarrow \lambda^{\prime\prime}(\tau_0+\beta),
\]
from which \eqref{second-derivative} follows simply upon taking $\beta=0$.
\end{proof}

\begin{proof}[Proof of Theorem~\ref{inf-lambda}]

Using a Taylor expansion of $\gamma(\beta,b) $ near $\beta=0$, it follows from Proposition~\ref{pprties-ld-beta-b} that there exist $\beta_0>0$ and $\epsilon_1>0$ such that
\begin{equation}\label{belowbeta0}
\lambda_0\leq b <\lambda_0+\epsilon_1~\&~|\beta|\leq \beta_0 \Rightarrow \gamma(\beta,b)> 0\,. 
\end{equation}
From the definition of $\gamma $ in \eqref{def-ld-beta-b} and the min-max principle, we get
\begin{equation}
\gamma(\beta,b)\geq \lambda\big(\xi(b)+ \beta\big)-b\,. 
\end{equation}
Since $\lambda''(\tau_0)>0$, we get by Taylor's formula the existence of $\epsilon_2\in(0,\epsilon_1)$ and $\delta\in(0,\frac{\beta_0}2)$ such that
$$z\not\in(\tau_0-\delta,\tau_0+\delta)\implies \lambda(z)\geq \lambda_0+\epsilon_2\,.$$
Since $\xi(b)\to\tau_0$ as $b\to\lambda_0$, there exists $\epsilon_3\in(0,\epsilon_2)$  such that
$$\lambda_0\leq b\leq \lambda_0+\epsilon_3\implies |\xi(b)-\tau_0|\leq \frac{\beta_0}2\,.$$
It is easy to see that, for $b\in[\lambda_0,\lambda_0+\epsilon_3]$ and $|\beta|\geq \frac{\beta_0}2$,  $\xi(b)+\beta\notin (\tau_0-\delta,\tau_0+\delta)$, and consequently
\[
 \gamma(\beta,b)\geq \lambda\big(\xi(b)+ \beta\big)-b\geq \lambda_0+\epsilon_2-b\geq 0.
 \]
This combined with \eqref{belowbeta0} finishes the proof of Theorem~\ref{inf-lambda}.
\end{proof}

\section{The model on  a half cylinder}\label{models-HC}
Recall that $S_{R}=(-R,R)\times \R$ and  ${{\bf A}_{\rm app}}$ is the magnetic potential introduced in \eqref{Aapp}. We introduce the space 
\begin{equation}\label{eq:D-per}
\mathcal{D}^{\rm per}= \Big\{   u\in L^2_{\rm loc}(\R^2)~:~ (\nabla -i{\bf A}_{\rm app})u \in L^2(S_R),~\exists\, z \in \R,~\quad u(x_1+2R, x_2)= 
e^{2 iz R}u(x_1,x_2) \Big\}\,,
\end{equation}
and  the ground state energy,
\begin{equation}\label{ebr-per}
\mathfrak{e}^{\rm per}(b;R)= {\rm inf}\left\{ \mathcal{E}_{R,b}(u)  ~:~  u\in \mathcal{D}^{\rm per}\right\}\,,
\end{equation}
where $\mathcal E_{R,b}$ is the functional in \eqref{E:Rb}.

For every $b>0$, let  $\xi(b)$ be as defined in Theorem~\ref{prop-min} and define the function
\begin{equation}\label{eq:psi-b}
\R^2\ni (x_1,x_2)\mapsto \psi_{b}(x_1,x_2)= e^{i\xi(b)x_1}f_{\xi(b),b}(x_2).
\end{equation}
We will prove 

\begin{thm}\label{lb-e-per}
There exists $\epsilon>0$ such that, for all $b \in [\lambda_0,\lambda_0+\epsilon)$ and $\psi\in \mathcal{D}^{\rm per}$,
\begin{equation}\label{lb-energy}
\mathcal{E}_{R,b} (\psi)\geq\mathcal{E}_{R,b} (\psi_b) .
\end{equation}
\end{thm}

\begin{rem}
It is easy to see that
\[
\psi_b(x_1+2R,x_2)= e^{2i\xi(b)R}e^{i\xi(b)x_1}f_{\xi(b),b}(x_2)=e^{2i\xi(b)R}\psi_b(x_1,x_2).
\]
Thus $\psi_b\in \mathcal{D}^{\rm per}$ (take $z=\xi(b)$). Consequently, we infer from Theorem~\ref{lb-e-per} that $\psi_b$ is the minimizer of $\mathcal{E}_{R,b}$ in $\mathcal{D}^{\rm per}$.
By \eqref{inf-formula} and invoking Theorem \ref{prop-min}, the minimal energy is:
\begin{equation}\label{min-energy}
\mathfrak e^{\rm per}(b;R)=\mathcal{E}_{R, b}(\psi_{b})= -{ bR}\|{f_{\xi(b),b}}\|_{4}^4= -{ b^{-1} R} \dfrac{(b-\lambda_0)^2 }{\|u_{0}\|_4^4}\big(1+g(b)\big),
\end{equation}
where $g(b)$ is independent of $R$ and satisfies $g(b)\to0$ as $b\searrow\lambda_0$.
\end{rem}

\begin{proof}[Proof of Theorem~\ref{lb-e-per}]
We follow the proof of Almog-Helffer \cite{Al.He.} devoted to the potential term $(t+\xi(b))^2$. 
Firstly, let us notice that the space
\begin{equation}\label{eq:D0-newspace}
\mathcal D_0=\{\psi\in\mathcal D^{\rm per}\cap C^\infty(\R^2)~:~\exists~M>0,~{\rm supp}\psi\subset \R\times [-M,M]\,\}\
\end{equation}
is dense in $\mathcal D_{\rm per}$, the space in \eqref{eq:D-per}, relative to the norm $\|u\|_{\mathcal D^{\rm per}}:=\|u\|_{L^2(S_R)}+\|(\nabla-i\Ab_{\rm app})u\|_{L^2(S_R)}$. So it is enough to prove \eqref{lb-energy} for $\psi\in\mathcal D_0$. The proof consists of four steps. Since $f_{\xi(b),b}>0$ in $\R_+$, we can represent the space $\mathcal D_0$ in the following useful form
\begin{multline}\label{eq:D0-newspace*}
\mathcal D_0=\{e^{izx_1}f_{\xi(b),b}(x_2)v(x_1,x_2)~:~z\in\R,~v\in C^\infty(\R^2)~{\rm is~}2R\text{-periodic in the variable }x_1\\
\&~\exists M>0,~{\rm supp}v\subset\R\times[-M,M] \}\,. 
\end{multline}

{\bf Step~1.}

Choose $b\in[\lambda_0,\lambda_0+\epsilon)$ so that Theorem~\ref{inf-lambda} holds.   Pick $\psi\in\mathcal D_0$ in the form (see \eqref{eq:D0-newspace*})
\begin{equation}\label{function1}
(x_1,x_2) \mapsto \psi (x_1,x_2) :=   e^{i\xi(b)x_1}  f_{\xi(b),b}(x_2)v\,,
\end{equation}
where $v(x_1,x_2)$ is smooth, vanishes for $|x_2|$ large enough, and periodic with respect to the first variable, i.e.
$v(x_1,x_2)= v(x_1+2R,x_2)$.

The following formula will allow us to  compare the energies of $\psi$ and $\psi_b$ (see \cite[Thm.~3.1, Eqs.~(3.5)-(3.7)]{Al.He.} for the detailed computations):
\begin{multline}\label{eq:combining}
\mathcal{E}_{R,b} (\psi)- \mathcal{E}_{R,b} (\psi_b)
=  \int_{\R} \int_{-R}^{R}\l\Bigg(f_{\xi(b),b}^2 |\nabla v|^2
          +2\left(\frac{ x_{2}^2}{2}+\xi(b)\right)f_{\xi(b),b}^2\Im(  v \partial_{x_1}v)\Bigg)dx_1 dx_2\\
     + \frac{b}{2}\int_{\R}\int_{-R}^{R}f_{\xi(b),b}^4(1-|v|^2)^2 dx_1dx_2.
\end{multline}
By periodicity we can expand $v$ in a  Fourier series as  follows
\[
v(x_1,x_2)= \sum_{n=-\infty}^{\infty}  v_n(x_2) e^{i n \frac{\pi}{R}x_1} 
\]
where
\begin{equation}\label{eq:revision-vn}
v_n(x_2)=\frac1{2R}\int_{-R}^R v(x_1,x_2)e^{-in\frac{\pi}{R}x_1}\,dx_1\,.\end{equation}
So, we can rewrite \[
\psi(x_1,x_2)= \displaystyle\sum_{n=-\infty}^{\infty} e^{i n \frac{\pi }{R}x_1}    e^{i \xi(b) x_1}   ( v_{n}f_{\xi(b),b})(x_2).
\]
Thus, the equation \eqref{eq:combining} reads as follows
\begin{equation}\label{delta-E}
\mathcal{E}_{R,b} (\psi)- \mathcal{E}_{R,b} (\psi_b)
 =  \sum_{n=-\infty}^{\infty} J(v_n; \frac{n\pi}{R})     + \frac{b}{2}\int_{\R}\int_{-R}^{R}f_{\xi(b),b}^4(1-|v|^2)^2 dx_1dx_2,
\end{equation}      
where
\[
J(v_n;\beta)=\int_{\R} |f_{\xi(b),b}|^2\left[ |v_n^\prime|^2+  \Big(\beta^2+2\beta\big(\frac{ x_{2}^2}{2}+\xi(b)\big)\Big)|v_n|^2                                \right] dx_2\,.
\]
It results from \eqref{eq:revision-vn} that $v_n(x_2)$ is a smooth function with compact support (since $v(x_1,x_2)$ is smooth and vanishes for $x_2$ large enough). Let $w_n(x_2)=f_{\xi(b),b}(x_2) v_n(x_2) $. It is easy to see that 
\[
\int_{\R} |f_{\xi(b),b}|^2 |v_n^\prime|^2  dx_2= \int_{\R}       \left[     -\left(\frac{w_n^2 f_{\xi(b),b}^\prime}{f_{\xi(b),b}}\right)^\prime  +  \frac{w_n^2 f_{\xi(b),b}^{\prime\prime}}{f_{\xi(b),b}}     +|w_n^\prime|^2        \right]dx_2\,,
\]
where, after an integration by parts,
\[
 \int_{\R}     \left(\frac{w_n^2 f_{\xi(b),b}^\prime}{f_{\xi(b),b}}\right)^\prime  dx_2=0  .
 \]
Consequently, using  the equation satisfied by $f_{\xi(b),b}$ in \eqref{eqt-f}, we get
$$
\int_{\R} |f_{\xi(b),b}|^2 |v_n^\prime|^2  dx_2= \int_{\R}       \left[   |w_n^\prime|^2   +\left(  \left(\frac{ x_{2}^2}{2}+\xi(b)\right)^2-b  (1- f_{\xi(b),b}^2)\right)
|w_n|^2     \right]dx_2.
$$
Now we insert this into the expression of $J(v_n;\beta)$ then use the min-max principle and get
\begin{align*}
J(v_n;\dfrac{n\pi}{R})&= \int_{\R}       \left[   |w_n^\prime|^2   +\Big(  \big( \frac{ x_{2}^2}{2}+\xi(b)+\frac{n\pi}{R}\big)^2-b  (1- f_{\xi(b),b}^2)                                                                    \Big  )
|w_n|^2     \right]dx_2\\
&\geq \gamma\left(\frac{n\pi}{R},b\right)\int_{\R}|w_n|^2 dx_2,
\end{align*}
where $\gamma(\cdot,b)$ was introduced in \eqref{def-ld-beta-b}. Note that $\gamma(\cdot,b)\geq 0$ by Theorem~\ref{inf-lambda}.
Inserting this into \eqref{delta-E}, we obtain
\begin{multline}\label{delta-E}
\mathcal{E}_{R,b} (\psi)- \mathcal{E}_{R,b} (\psi_b)\\
\geq  2R \gamma\left(\frac{n\pi}{R},b\right) \sum_{n=-\infty}^{\infty}  \int_{\R}|f_{\xi(b),b} (x_2)v_n(x_2)|^2 dx_2           + \frac{b}{2}\int_{\R}\int_{-R}^{R}f_{\xi(b),b}^4(1-|v|^2)^2 dx_1dx_2\geq 0.
\end{multline}

{\bf Step~2. }

Now we  consider an arbitrary function  $\psi\in\mathcal{D}_0$  which can be expressed in the form (see \eqref{eq:D0-newspace*})
\begin{equation}\label{function2}
\psi(x_1,x_2)=e^{iz x_1 } f_{\xi(b),b}(x_2) v(x_1,x_2)\,.
\end{equation}
Note that in \eqref{function1}, we handled  the special case $z=\xi(b)$. Here 
we assume that :
\begin{equation}\label{Eq:Asp-on-R}
\frac{R}{\pi}(z-\xi(b))= \frac{r}{s}, 
\end{equation}
for some $(r,s)\in \mathbb{Z}\times\mathbb{N}$.
We can rewrite $\psi$ as
\[
\psi(x_1,x_2)=e^{i \xi(b) x_1 } f_{\xi(b),b}(x_2) v^{\rm per}(x_1,x_2),
\]
where $v^{\rm per}(x_1,x_2):=e^{i(z-\xi(b))x_1} v(x_1,x_2) $. 

The function $v^{\rm per}$ is $2sR$-periodic with respect to the first variable. Thus 
 $\psi$ falls in the case studied in Step~1 but  with $R$ replaced by $sR$ and $s\in \mathbb{N}$. We apply the conclusion in Step~1 and write 
\[
\mathcal{E}_{sR,b}(\psi)\geq \mathcal{E}_{sR,b}(\psi_b).
\]
Next we observe that, for $s\in\mathbb{N}$, 
\[
\mathcal{E}_{sR,b}(\psi)= s \mathcal{E}_{R,b}(\psi) \quad {\rm and}\quad \mathcal{E}_{sR,b}(\psi_b)= s \mathcal{E}_{R,b}(\psi_b).
\]
So we deduce that
\[
\mathcal{E}_{R,b}(\psi)\geq \mathcal{E}_{R,b}(\psi_b),
\]
for all $\psi\in \mathcal{D}_0$ but under the condition in \eqref{Eq:Asp-on-R}.

{\bf Step~3.}

The general result follows from the density of rational numbers in $\R$. We present the details for the sake of convenience. Pick  $z\in\R$ and an arbitrary smooth function 
$\psi(\cdot;z)\in\mathcal D_0$ having the form (see \eqref{eq:D0-newspace*})
\[
\psi(x_1,x_2;z):= e^{iz x_1 } f_{\xi(b),b}(x_2) v(x_1,x_2)\,.\]
We will prove that 
\begin{equation}\label{eq:conc*}
\mathcal{E}_{R,b}(\psi(\cdot;z))\geq \mathcal{E}_{R,b}(\psi_{b}),
\end{equation}
which yields the desired result.

Define $\alpha\in\R$ as follows
\[
\frac{R}{\pi}(z-\xi(b))=\alpha\in \R.
\]
Let $\alpha_{n}=\dfrac{[n\alpha]}{n}\in \mathbb{Q}$, where $[\cdot]$ denotes the integer part. It is clear that $\alpha_n\rightarrow \alpha$ in $\R$. Define the sequence $z_n$ as follows
\[
\frac{R}{\pi}(z_n-\xi(b))=\alpha_n \in \mathbb{Q}.
\]
We apply the conclusion in Step 2 with $z_{n}$, it follows that
\begin{equation}\label{particular-case}
\mathcal{E}_{R,b}(\psi(\cdot;z_n))\geq \mathcal{E}_{R,b}(\psi_{b}).
\end{equation}
It is clear that $z_n\rightarrow z$. From this, we deduce that $\mathcal{E}_{R,b}(\psi(\cdot;z_n))\rightarrow \mathcal{E}_{R,b}(\psi(\cdot;z))$. Since $\mathcal{E}_{R,b}(\psi_{b})$ is independent of $z$, taking the limit in \eqref{particular-case} yields \eqref{eq:conc*}.
\end{proof}

\section{Proof of Theorem~\ref{main-theorem}}\label{Proof-MT}

Recall the ground state energies $\mathfrak{e}$ and $\mathfrak{e^{\rm per}}$ from \eqref{ebr} and \eqref{ebr-per} respectively. We decompose the proof of Theorem~\ref{main-theorem} into  two steps.

{\bf Step~1: Lower bound.}

Since every function in $H^1_0(S_R)$ can be extended by periodicity to a function in the domain $\mathcal{D}^{\rm per}$, we get immediately that, for all $L,R>0$, 
\begin{equation}\label{compare}
\mathfrak{e}(L^{-3/2};R)\geq \mathfrak{e}^{\rm per}(L^{-3/2};R).
\end{equation}
Now, Theorem~\ref{lb-e-per} and the formula in \eqref{min-energy} give us, for all $L,R>0$,  
\[
\frac{\mathfrak{e}^{\rm per}(L^{-3/2};R)}{2R}\geq  -\frac{ L^{2/3}}{2} \dfrac{(L^{-2/3}-\lambda_0)^2 }{\|u_{0}\|_4^4}\big(1+\mathfrak g(L)\big),
\]
where $\mathfrak g(L)$ is independent of $R$ and tends to $0$ as $L\nearrow\lambda_0^{-3/2}$. 
Thus \eqref{compare} yields 
\[
\frac{\mathfrak{e}(L^{-3/2};R)}{2R}\geq  -\frac{ L^{2/3}}{2} \dfrac{(L^{-2/3}-\lambda_0)^2 }{\|u_{0}\|_4^4}\big(1+\mathfrak g(L)\big)\,.
\]
In light of Theorem~\ref{Thm:EL}, we get the desired lower bound upon  taking ${R\rightarrow\infty}$.

{\bf Step~2: Upper bound.}

To get an upper bound, we need to use a suitable test configuration. Let $\theta_R\in C_{c}^\infty(\R)$ be a function satisfying,  
\[
{\rm supp}~\theta_R\subset (-R,R),\quad 0\leq \theta_R\leq 1,\quad \theta_R=1 \quad {\rm in}\quad (-R+1,R-1),
\]
and 
\[
|\theta'_R|\leq C\,,
\]
where $C>0$ is a universal constant.

 We introduce 
\[
\psi(x_1,x_2)= e^{i\xi(L^{-2/3})x_1}f_L(x_2)  \theta_{R}(x_1).
\]
where 
\[
f_{L}(x_2):=f_{\xi(L^{-2/3}),L^{-2/3}}(x_2).
\]
Here, we recall  $\xi(b)$ and $f_{\xi(b),b}$ from Theorems~\ref{prop-min} and \ref{GLF-M} respectively.

We start by estimating 
\begin{equation}
 \mathcal{E}_{R,L^{-2/3}}(\psi)= \int_{S_R}\left(|(\nabla -i {\bf A_{\rm app}})\psi|^2dx-L^{-2/3}|\psi|^2 dx+\frac{L^{-2/3}}{2}|\psi|^4\right)dx.
   \end{equation}
An integration by parts yields,
\begin{multline}\label{Eq:lb1}
\int_{S_R}|(\nabla -i {\bf A_{\rm app}})\psi|^2dx =  \left\langle     \theta_R^2(x_1) f_{L}(x_2), -(\nabla -i {\bf A_{\rm app}})^2 e^{i\xi(L^{-2/3})x_1} f_{L}(x_2)                                                                               \right\rangle          \\
   +\int_{S_R}|f_{L}(x_2)\theta^{\prime}_R(x_1)|^2 dx.                                    
  \end{multline}
  Note that
  \begin{equation}\label{Eq:lb2}
  \begin{aligned}
& \left\langle     \theta_R^2(x_1) f_{L}(x_2), -(\nabla -i {\bf A_{\rm app}})^2 e^{i\xi(L^{-2/3})x_1} f_{L}(x_2)                                                                               \right\rangle\\
&\qquad =  \int_{S_R} \theta_R^2(x_1)\left(|f_{L}^\prime(x_2)|^2 +\Big(\frac{x_2^2}{2}+\xi(L^{-2/3})\Big)^2 |f_{L}(x_2)|^2\right)dx_1 dx_2  \\
 &\qquad\leq  2R \int_{S_R}\left(|f_{L}^\prime(x_2)|^2 +\Big(\frac{x_2^2}{2}+\xi(L^{-2/3})\Big)^2 |f_{L}(x_2)|^2  \right)dx_2.
      \end{aligned}
      \end{equation}
By the construction of $\theta_R$, we have that ${\rm supp}~\theta^{\prime}_R\subset[-R+1,R-1]$ and $|\theta'_R|\leq C$. Thus 
\begin{equation}\label{Eq:lb3}
\int_{S_R}|f_{L}(x_2)\theta^{\prime}_R(x_1)|^2 dx_1dx_2  =\int_{-R+1}^{R-1}|\theta^{\prime}_R(x_1)|^2 dx_1\int_{\R}|f_{L}(x_2)|^2 dx_2\leq C\|f_L\|_2^2.
\end{equation}
Here $\|f_L\|_2<\infty$ but depends on $L$. Substituting  \eqref{Eq:lb2} and \eqref{Eq:lb3} in \eqref{Eq:lb1}, we find
\begin{equation}
\int_{S_R}|(\nabla -i {\bf A_{\rm app}})\psi|^2dx\leq    2R\left( \int_{\R}|f_{L}^\prime(x_2)|^2 +\Big(\frac{x_2^2}{2}+\xi(L^{-2/3})\Big)^2 |f_{L}(x_2)|^2\right)  dx_2+ C\|f_L\|_2^2\,.
\end{equation}
We have the following decomposition,
\begin{align*}
\int_{S_{R}}|\psi|^2dx&= \int_{S_{R}}\theta_R(x_1)^2 |f_{L}(x_2)|^2dx_1 dx_2\\
&=2R \int_{\R} |f_{L}(x_2)|^2dx_2-\int_{S_R}(1-\theta_R^2(x_1))|f_{L}(x_2)|^2dx_1 dx_2.
\end{align*}
Again, the assumption on the support of $\theta_R$ yields
\begin{equation}
\int_{S_R}(1-\theta_R^2(x_1))|f_{L}(x_2)|^2dx_1 dx_2\leq 2\|f_L\|_2^2\,.
\end{equation}
Consequently, we obtain, for all $R>2$,
\begin{equation}\label{18}
\begin{aligned}
\mathfrak{e}(L^{-2/3};R)&\leq \mathcal{E}_{R,L^{-2/3}}(\psi)\\
&\leq  2R \int_{\R}\left(   |f_{L}^\prime|^2 +\Big(\frac{x_2^2}{2}+\xi(L^{-2/3})\Big)^2 |f_{L}|^2  
 -L^{-2/3}    |f_{L}|^2  +\dfrac{L^{-2/3}}{2}    |f_{L}|^4 
\right)  dx_2\\
&\qquad+\max(C,2)\|f_L\|_2^2\,.
\end{aligned}
\end{equation}
Since $f_L$ is a minimizer of the functional \eqref{GLF-M*} for $(\alpha,b)=\big(\xi(L^{-2/3}),L^{-2/3}\big)$, \eqref{18} reads 
\begin{equation}
\mathfrak{e}(L^{-2/3};R)\leq  2R \,\mathfrak{b}\big(\xi(L^{-2/3}),L^{-2/3}\big)+\max(C,2)\|f_L\|_2^2\,, 
\end{equation}
where $\mathfrak{b}$ was introduced in \eqref{GLF-M:gse}.

Dividing by $2R$, we get 
\begin{equation}
\dfrac{\mathfrak{e}(L^{-2/3};R)}{2R}\leq  \mathfrak{b}\big(\xi(L^{-2/3}),L^{-2/3}\big)+\dfrac{\max(C,2)\|f_L\|_2^2}{R}\,.
\end{equation}
 Taking $\limsup_{R\rightarrow \infty}$ on both sides and invoking Theorem~\ref{Thm:EL}, we infer that, for all $L>0$,
\begin{equation}\label{1prime}
E(L)=\limsup_{R\rightarrow\infty}\dfrac{\mathfrak{e}(L^{-2/3};R)}{2R}\leq   \mathfrak{b}\big(\xi(L^{-2/3}),L^{-2/3}\big).
\end{equation}
In view of Theorem~\ref{prop-min}, we see that, as $L\nearrow \lambda_0^{-3/2}$,
\[
 \mathfrak{b}\Big(\xi(L^{-2/3}),L^{-2/3}\Big)=   - {\frac{L^{2/3}}{2}}\dfrac{(L^{-2/3}-\lambda_0)^2 }{\|u_{0}\|_4^4}\big(1+o(1)\big)\,.               
\]
Inserting this into \eqref{1prime}, we get, as $L\nearrow \lambda_0^{-3/2}$,
\[
E(L)\leq  - \frac{L^{2/3}}{2} \dfrac{(L^{-2/3}-\lambda_0)^2 }{\|u_{0}\|_4^4}(1+o(1))\,.
\]

\section{Proof of Theorem~\ref{corol:KN1}}\label{sec:proof-KN1}

We will improve the estimate in \eqref{eq:ae-HK} by providing an explicit control of the remainder term. We will do this by  carefully examining the upper and lower bounds obtained in \cite{HK1}.

To simplify the presentation, we will assume that the set $\Gamma$ (introduced in \eqref{eq:Gamma*}) consists of a single smooth curve. 
When $\Gamma$ consists of a finite number of components, we can apply the analysis in this section to each component separately and sum up the results.

We will use the following notation:
\begin{itemize}
\item $ds$ denotes the arc-length measure on $\Gamma$\,;
\item $|\Gamma|=\int_\Gamma ds(x)$ denotes the arc-length measure of $\Gamma$\,;
\item ${\rm dist}_{\Gamma}:\Gamma\times\Gamma\to[0,\infty)$ denotes the arc-length distance in $\Gamma$\,.
\end{itemize}

We begin with the following geometric lemma.

\begin{lem}\label{lem:ball}
There exist two positive constants $C$ and $\ell_0$ (which depend on the domain $\Omega$, the function $B_0$ and the set $\Gamma$ in \eqref{eq:Gamma*}) such that, for all $a\in\Gamma$ and $\ell\in(0,\ell_0)$ satisfying
$$\overline{D(a,\ell)}\subset\Omega$$
then
$$\left|\int_{\overline{D(a,\ell)}\cap\Gamma} ds(x)-2\ell\right|\leq C\ell^2\,.$$
\end{lem}
\begin{proof}
Let $a\in \Gamma$ and $\ell>0$ such that $\overline{D(a,\ell)}\subset\Omega$. By a translation, we may assume that $a=(0,0)$. We can select an interval $I_a$, a $C^2$ function  $u_a:I_a\to\R$, and a constant $\tilde C>0$ such that
$$\overline{D(a,\ell)}\cap\Gamma\subset\{(s,u_a(s))~:~s\in I_a\}\,,\quad 0\in I_a, \quad (0,u_a(0))=0\,,$$
and
$$\forall~s\in I_a\,,\quad |u_a(s)|+|u_a'(s)|+ |u_a''(s)|\leq \tilde C\,.$$
Furthermore,  by the compactness of the set $\Gamma$, we may assume that the constant $\tilde C$ is independent of $a$ and $\ell$, for $\ell$ sufficiently small.

Define the function $f(s)=s^2+\big(u_a(s)\big)^2-\ell^2$. Using Taylor's formula for the function $u_a$ near $0$,  we can prove the following, for $\ell$ sufficiently small:
\begin{itemize}
\item There exist $s_1\in(-2\ell,0)$ and $s_2\in(0,2\ell)$ such that $f(s_1)=f(s_2)=0$ (by the intermediate value theorem)\,;
\item $f'(s)>0$ on $(-2\ell,2\ell)$\,;
\item $s_1$ and $s_2$ are the unique zeros of the function $f$ on the interval $(-2\ell,2\ell)$\,;
\item $s_1$ and $s_2$ satisfy
$$s_1=\frac{-\ell}{\sqrt{1+|u_a'(0)|^2}}+\mathcal O(\ell^2)\quad{\rm and}\quad s_2=\frac{\ell}{\sqrt{1+|u_a'(0)|^2}}+\mathcal O(\ell^2)\,.$$  
\end{itemize}
Therefore, we deduce that $\overline{D(a,\ell)}\cap \Gamma=\{(s,u_a(s))~:~s_1\leq s\leq s_2\}$ and
$$\int_{\overline{D(a,\ell)}\cap \Gamma}ds(x)=\int_{s_1}^{s_2}\sqrt{1+|u'_a(s)|^2}\,ds=2\ell+\mathcal O(\ell^2)~{\rm as}~\ell\to0_+\,.$$
\end{proof}

With Lemma~\ref{lem:ball} in hand, we can a construct a covering of $\Gamma$ by disks with disjoint interior.

\begin{lem}\label{lem:construction-balls}
There exist two positive constants $C$ and $\ell_0$ such that, for all  $\ell\in(0,\ell_0)$,  there exist $N\in\mathbb N$ and a collection of points $(a_j)_{1\leq j\leq N}$ on $\Gamma$ such that
\begin{align*}
&\forall~j,\quad \Big|{\rm dist}_{\Gamma}(a_j,a_{j+1})-2\ell\Big|\leq C\ell^2~\&~D(a_j,\ell)\subset\Omega\,,\\
&D(a_j,\ell)\cap D(a_{j'},\ell)=\emptyset~{\rm for~}j\not=j'\,,\\
&\left| N-\frac{|\Gamma|}{2\ell}\right|\leq C\,.
\end{align*}
\end{lem}
\begin{proof}
For all $\ell\in(0,1)$, let $\mathfrak n$ be the unique natural number satisfying
$$
\frac{|\Gamma|}{2\ell}\left(1+\frac\ell2\right)^{-1}-1\leq \mathfrak n<\frac{|\Gamma|}{2\ell}\left(1+\frac{\ell}2\right)^{-1}\,.
$$
We select a collection of points $(b_j)_{1\leq j\leq\mathfrak n}\subset\Gamma$ such that ${\rm dist}_{\Gamma}(b_j,b_{j+1})=\frac{|\Gamma|}{\mathfrak n}$. For all $j$, let $e_j=|b_{j+1}-b_j|$ be  the Euclidean distance  between the points $b_{j+1}$ and $b_j$. We define the number $N$ as follows
$$N={\rm Card}\mathcal J\quad{\rm where~} \mathcal J=\{j~:~D(b_j,e_j)\subset\Omega\}\,.$$
For $\ell$ sufficiently small, we get that $\mathcal  J=\{j_0+k~:~1\leq k\leq N\}$ for some $j_0\in\{1,\cdots,\mathfrak n\}$. 
Now, for all $k\in\{1,\cdots, N\}$, we set $a_k=b_{j_0+k}$.

The points $(a_k)$ and the number $N$ satisfy the properties mentioned in Lemma~\ref{lem:construction-balls}. 
The details can be found in \cite[Proof of Lemma~5.2, Step~2]{HK1}.
\end{proof}

In Lemma~\ref{lem:ub-ball} below,  $\Fb$ denotes the unique vector field satisfying
\begin{equation}\label{eq:F}
\curl\Fb=B_0\,,\quad{\rm div}\Fb=0\quad{\rm in~}\Omega\,,\quad \nu\cdot\Fb=0~{\rm on~}\partial\Omega\,,
\end{equation}
where $\nu$ is the unit normal vector of the boundary of $\Omega$. Also, we introduce the following local Ginzburg-Landau energy
\begin{equation}\label{eq:E0}
\mathcal E_0(u,A;U)=\int_U\left(|(\nabla-i\kappa H A)u|^2-\kappa^2|u|^2+\frac{\kappa^2}2|u|^4\right)\,dx\,,
\end{equation}
where $U$ is an open subset of $\R^2$.

\begin{lem}\label{lem:ub-ball}
Let $0<M_1<M_2$. There exist two positive constants $C$ and $\kappa_0$ such that the following is true.

Assume that
\begin{itemize}
\item $\kappa\geq \kappa_0$ and $M_1\kappa^2\leq H\leq M_2\kappa^2$\,;
\item $\ell=\kappa^{-7/8}$, $\mathfrak a\in\Gamma$  and $D(\mathfrak a,\ell)\subset\Omega$\,;
\item $\mathfrak x\in\overline{D(\mathfrak a,\ell)}\cap\Gamma$ and $L_{\mathfrak x}=|\nabla B_0(\mathfrak x)|\frac{H}{\kappa^2}$\,.
\end{itemize}
Then there exists a function $w_{\mathfrak a,\mathfrak x}\in H^1_0(D(a,\ell))$ such that
$$\mathcal E_0\big(w_{\mathfrak a,\mathfrak x},\Fb;D(\mathfrak a,\ell)\big)\leq \Big(2L_{\mathfrak x}^{1/3}E(L_{\mathfrak x})
+C\kappa^{-1/16}\Big)\kappa\ell\,,$$
where the function $E(\cdot)$ is introduced in  \eqref{eq:E(L)}.
\end{lem}
\begin{proof}
We will skip the reference to the points $\mathfrak a$ and $\mathfrak x$ by writing $L=L_{\mathfrak x}$ and $w=w_{\mathfrak a,\mathfrak x}$. Define $a=A(\kappa\ell)^{-1}$ and $R=L^{1/3}\kappa\ell$, where $A$ is a constant selected such that, for $\kappa$ sufficiently large, we have
\begin{equation}\label{eq:cond-R}
R\geq 4\max(a^{-1/2}L^{-2/3},1)\,.
\end{equation}
Then we take $w$ as in \cite[Eq.~(5.11)]{HK1}. Since $R$ satisfies \eqref{eq:cond-R}, then the function $w_{\mathfrak a}$ satisfies (see \cite[Eq.~(5.15)]{HK1}), for some constant $\tilde C>0$ and for all $\delta>0$,
\begin{multline*}
\mathcal E_0\big(w,\Fb;D(\mathfrak a,\ell)\big)\leq 2(1+\delta)(1-a)R\,E(L)\\ +\tilde C\Big((1+L^{-2/3})R^{1/3}
+a^{-1/2}(1+a^{-1}L^{-2/3}R^{-2})+(\delta\kappa^2+\delta^{-1}\kappa^2H^2\ell^6)\ell^2\Big)\,.
\end{multline*}
For $\delta=\kappa^{-3/8}$, $\ell=\kappa^{-7/8}$, $a\approx(\kappa\ell)^{-1}$ and $H\approx \kappa^2$, we get the upper bound 
in Lemma~\ref{lem:ub-ball}, for some constant $C>\tilde C$.
\end{proof}

Now we can prove the

\begin{proposition}\label{prop:ub-HK}
Let $0<M_1<M_2$. There exist two positive constants $C$ and $\kappa_0$ such that, for all  $\kappa\geq \kappa_0$ and $M_1\kappa^2\leq H\leq M_2\kappa^2$, the ground state energy in \eqref{egs} satisfies
$${\mathrm E}_{\rm gs}(\kappa,H)\leq \kappa\int_\Gamma\left(\frac{H}{\kappa^2}|\nabla B_0(x)|\right)^{1/3} E\left(\frac{H}{\kappa^2}|\nabla B_0(x)|\right)\,ds(x)+C\kappa^{15/16}\,.$$
\end{proposition}
\begin{proof}
Let $\ell=\kappa^{-7/8}$ and $\big(D(a_j,\ell)\big)_{1\leq j\leq N}$ be the collection of the pairwise disjoint disks constructed in  Lemma~\ref{lem:construction-balls}, for $\kappa$ sufficiently large. For all $j$, choose the point $x_j$ such that
$$\min_{x\in\overline{D(a_j,\ell)}\cap\Gamma}\left(\frac{H}{\kappa^2}|\nabla B_0(x)|\right)^{1/3} E\left(\frac{H}{\kappa^2}|\nabla B_0(x)|\right)=\left(\frac{H}{\kappa^2}|\nabla B_0(x_j)|\right)^{1/3} E\left(\frac{H}{\kappa^2}|\nabla B_0(x_j)|\right)\,.$$
We define the function $w\in H^1_0(\Omega)$ as follows
$$w(x)=
\begin{cases}
w_{a_j,x_j}(x)~:~x\in D(a_j,\ell)\\
0~:~x\not\in \displaystyle\bigcup_{1\leq j\leq N}D(a_j,\ell)\,.
\end{cases}
$$
Let $\Fb$ be the vector field in \eqref{eq:F}. Since ${\rm E}_{gs}(\kappa,H)\leq\mathcal E(w,\Fb)=\displaystyle\sum_{j=1}^{N}\mathcal E_0(w_{a_j,x_j},\Fb)$, 
 Lemma~\ref{lem:ub-ball} yields
\begin{align*}
{\rm E}_{\rm gs}(\kappa,H)
&\leq \sum_{j=1}^{N}\left(\Big(2L_{x_j}^{1/3}E(L_{x_j})
+C\kappa^{-1/16}\Big)\kappa\ell\right)\\
&\leq \kappa\underset{\rm lower~ Riemann ~sum}{\underbrace{\sum_{j=1}^{N}\left(\Big(|\overline{D(a_j,\ell)}\cap\Gamma| L_{x_j}^{1/3}E(L_{x_j})\right)}}
+C\Big(\ell+\kappa^{-1/16}\Big)\kappa{\rm\quad by ~Lemma~\ref{lem:ball}}\\
&\leq \kappa\int_{V_\ell} \left\{\left(\frac{H}{\kappa^2}|\nabla B_0(x)|\right)^{1/3} E\left(\frac{H}{\kappa^2}|\nabla B_0(x)|\right)\right\}\,ds(x)+C\kappa^{15/16}\,,
\end{align*}
where $V_\ell=\displaystyle\bigcup_{j=1}^{N}D(x_j,\ell)\cap \Gamma$. But, by Lemma~\ref{lem:construction-balls}, 
$|\Gamma\setminus V_\ell|\leq C\ell$ which is what we need to  obtain the upper bound in Proposition~\ref{prop:ub-HK}.
\end{proof}

\begin{proposition}\label{prop:lb-HK}
Let $0<M_1<M_2$. There exist two positive constants $C$ and $\kappa_0$ such that, for all  $\kappa\geq \kappa_0$ and $M_1\kappa^2\leq H\leq M_2\kappa^2$, the ground state energy in \eqref{egs} satisfies
$${\mathrm E}_{\rm gs}(\kappa,H)\geq \kappa\int_\Gamma\left(\frac{H}{\kappa^2}|\nabla B_0(x)|\right)^{1/3} E\left(\frac{H}{\kappa^2}|\nabla B_0(x)|\right)\,ds(x)-C\kappa^{11/12}\,.$$
\end{proposition}
\begin{proof}
Let $a>0$ and $\delta>0$ be two sufficiently small parameters. Let $\ell=\delta H^{-1/3}$ and define the two domains
$$D_1=\{x\in\Omega~:~{\rm dist}(x,\Gamma)<2\sqrt{a}\ell\}\quad{\rm and}\quad D_2=\{x\in\Omega~:~{\rm dist}(x,\Gamma)>\sqrt{a}\ell\}\,.$$
There exist two smooth functions $\chi_1$ and $\chi_2$ such that
$$\chi_1^2+\chi_2^2=1\,,\quad{\rm supp}\chi_j\subset D_j\,,\quad{\rm and}\quad |\nabla\chi_j|\leq C(a\ell^2)^{-1}\,,$$
for some positive constant $C$.

Let $(\psi,\Ab)$ be a minimizer of the functional in \eqref{GL-Energy}. The following holds (see \cite[Eq.~(7.11)]{HK1})
\begin{align*}{\rm E}_{\rm gs}(\kappa,H)&=\mathcal E(\psi,\Ab)\geq\mathcal E_0(\psi,\Ab;\Omega)\\
&\geq \sum_{j=1}^2\mathcal E_0(\chi_j\psi,\Ab;\Omega)-\frac{C}{\sqrt{a}\ell}\,,\end{align*}
where the functionals $\mathcal E$ and $\mathcal E_0$ are introduced in \eqref{GL-Energy} and \eqref{eq:E0} respectively.

We will select the parameters $a$ and $\delta$ such that $\sqrt{a}\ell\gg\kappa^{-1}$ (recall that $\ell=\delta H^{-1/3}$). By \cite[Thm.~6.3]{HK1},  $|\psi|^2$ is exponentially small in $D_2$, hence $\mathcal E_0(\chi_2\psi,\Ab;\Omega)\geq -\kappa^{-1}$ for $\kappa$ sufficiently large. Consequently
\begin{equation}\label{eq:HK-lb}
{\rm E}_{\rm gs}(\kappa,H)
\geq \mathcal E_0(\chi_1\psi,\Ab;\Omega)-C\left(\frac{1}{\sqrt{a}\ell}+\frac1\kappa\right)\,.
\end{equation}
Having Lemma~\ref{lem:ball} in hand, we can use the following lower bound  (see \cite[Eq.~(7.19)]{HK1})
\begin{multline*}
\mathcal E_0(\chi_1\psi,\Ab;\Omega)\geq \kappa\int_{\Gamma}\left(\frac{H}{\kappa^2}|\nabla B_0(x)|\right)^{1/3} E\left(\frac{H}{\kappa^2}|\nabla B_0(x)|\right)\,ds(x)\\-C(\sqrt{a}\ell)^{-1}-C\Big(a+\delta+\delta^{2\alpha-1}\frac{\kappa}{H}H^{-2\alpha/3}+\eta \Big)\kappa\end{multline*}
for $\eta=\ell$ and for all  $\alpha\in(0,1)$. We insert this lower bound into \eqref{eq:HK-lb} then we choose $a=\delta=\kappa^{-1/6}$ and $\alpha=3/4$. This finishes the proof of Proposition~\ref{prop:lb-HK}.
\end{proof}

\begin{proof}[Proof of Theorem~\ref{corol:KN1}]
Propositions~\ref{prop:ub-HK} and \ref{prop:lb-HK} yield that
\begin{equation}\label{eq:corol:KN1*}
{\mathrm E}_{\rm gs}(\kappa,H)=\kappa\int_\Gamma\left(\frac{H}{\kappa^2}|\nabla B_0(x)|\right)^{1/3} E\left(\frac{H}{\kappa^2}|\nabla B_0(x)|\right)\,ds(x)+\mathcal O(\kappa^{11/12})\,.
\end{equation}
Under the assumption~\ref{ass:B0-2}, the principal term in \eqref{eq:corol:KN1*} satisfies \eqref{eq:ae-HK*} and is of order $|\Gamma_\kappa|\big(\rho(\kappa)\big)^2\geq c\big(\rho(\kappa)\big)^{5/2}$, for some constant $c>0$. By \eqref{eq:ae-HK*} and the assumption $\rho(\kappa)\gg \kappa^{-1/30}$, we get 
\begin{equation}\label{eq:corol:KN1}
\kappa^{11/12}\ll \frac{\kappa\lambda_0^{-3/2}}{2\|u_{0}\|_4^4} \int_{\Gamma}
 \left(\Big(\frac{H}{\kappa^2}|\nabla B_0(x)|\Big)^{-2/3}-\lambda_0\right)_+^2\,ds(x)\,.
\end{equation}
Now, collecting \eqref{eq:ae-HK*}, \eqref{eq:corol:KN1} and \eqref{eq:corol:KN1*}, we finish the proof of Theorem~\ref{corol:KN1}.  
\end{proof}

\subsection*{Acknowledgments}
The authors would like to thank B. Helffer for his valuable comments on the manuscript, and the anonymous referee for the valuable suggestions.  A.K. is supported by a grant from Lebanese University.


\begin{thebibliography}{100}  
   \bibitem{Al.He.}  Y. Almog and B. Helffer. The distribution of surface superconductivity along the boundary : on a conjecture of X. B. Pan. {\it SIAM J. Math. Anal.} {\bf 38}, 1715-1732 (2007).
 \bibitem{AHP}  Y. Almog, B. Helffer and X. B. Pan. Mixed normal-superconducting states in the presence of strong electric currents. {\it Arch. Rational Mech. Anal.} {\bf 223}, 419-462 (2017).
  %
     \bibitem{AK} W. Assaad and A. Kachmar. The influence of magnetic steps on bulk superconductivity. {\it Discrete and Continuous Dynamical Systems (A)} {\bf 36} (12), 6623-6643 (2016).  
 \bibitem{Att1} K. Attar. The ground state energy of the two dimensional Ginzburg-Landau functional with variable magnetic field. {\it Annales de l'Institut Henri Poincar\'e- Analyse Non-Lin\'eaire} {\bf 32}, 325-345 (2015).
    \bibitem{Att2} K. Attar. Energy and vorticity of the Ginzburg-Landau model with variable magnetic field. {\it Asymptot. Anal.} {\bf 93}, 75-114 (2015).
      \bibitem{Att3} K. Attar. Pinning with a variable magnetic field of the two-dimensional Ginzburg-Landau model. { \it Non-Linear Analysis: TMA.} {\bf 139}, 1-54 (2016).
      %
          \bibitem{CL} A. Contreras and X. Lamy. Persistence of superconductivity in thin shells beyond $H_{c1}$.  {\it Commun. Contemp. Math.} {\bf 18} article no. 1550047, 21 p, (2016).
          \bibitem{CR2}M. Correggi and N. Rougerie. Boundary behavior of the Ginzburg-Landau order parameter in the surface superconductivity regime. {\it Arch. Rational Mech. Anal}. {\bf 219}, 553-606 (2015).
\bibitem{CR1} M. Correggi and N. Rougerie. Effects of boundary curvature on surface superconductivity.
{\it Letters in Mathematical Physics}  1-23 (2016).
\bibitem{CR3} M. Correggi and Nicolas Rougerie. On the Ginzburg-Landau functional in the surface superconductivity regime. {\it Comm. Math. Phys.} {\bf 332}, 1297-1343 (2014).
 \bibitem{dGe} P.G. de\,Gennes. {\it Superconductivity of Metals and Alloys.} Benjamin, Amsterdam (1996).
        \bibitem{FH-cvpde} S. Fournais and B. Helffer. Energy asymptotics for type II superconductors. {\it  Calc. Var. Partial Differential Equations} {\bf 24} (3), 341-376 (2005).
   \bibitem{FH-b} S. Fournais and B. Helffer. {\it{Spectral methods in surface superconductivity}}. Progress in Nonlinear Differential Equations and Their
        Applications, Vol. 77. Birkh\"{a}user Boston Inc., Boston, MA (2010).
               \bibitem{FK3D} S. Fournais and A. Kachmar. The ground state energy of the three dimensional  Ginzburg-Landau functional. Part~I: Bulk regime. {\it Comm. Partial. Differential Equations} {\bf 38} (2), 339-383 (2013).
         \bibitem{FK-am} S. Fournais and A. Kachmar. Nucleation of bulk superconductivity close to critical magnetic field. {\it Advan. Math.} {\bf 226} (2), 1213-1258 (2011).  
           \bibitem{B.H.notes} B. Helffer. The Montgomery operator revisited.{ \it Colloquium Mathematicum} {\bf 118} (2), 391-400 (2011).
 \bibitem{HK} B. Helffer and A. Kachmar. Decay of superconductivity away from the magnetic zero set. {\it arXiv:1604.02402v1} (2016). 
 \bibitem{HK2} B. Helffer and A. Kachmar. From constant to non-degenerately vanishing magnetic fields in superconductivity. {\it Annales de l'Institut Henri Poincar\'e- Analyse Non-Lin\'eaire} {\bf 34}, 423-438 (2017).
     \bibitem{HK1} B. Helffer and A. Kachmar. The Ginzburg-Landau functional with vanishing magnetic field. {\it Arch. Rational Mech. Anal.} {\bf 218}, 55-122 (2015).
    \bibitem{HKRV} B. Helffer, Y. Kordyukov, N. Raymond, S. V\~{u}\,Ng\c{o}c. Magnetic wells in dimension three. {\it Analysis and PDE} {\bf 9} (7) 1575-1608 (2016).
    \bibitem{H.M.} B. Helffer and A. Mohamed. Semi classical analysis for the ground state energy of a Schr\"{o}dinger operator with magnetic wells. {\it J. Funct. Anal} {\bf 138} (1), 40-81 (1996).  

   \bibitem{Mon} R. Montgomery. Hearing the zero locus of a magnetic field. {\it Commun. Math. Phys.} {\bf 168} (3), 651-675 (1995). 
   \bibitem{P.K.} X. B. Pan and H. Kwek. Schr\"odinger operators with non-degenerately vanishing magnetic fields in bounded domains. {\it Trans. Am. Math. Soc.} {\bf 354} (10), 4201-4227 (2002).
   \bibitem{SS-b} E. Sandier and S. Serfaty. Vortices for the magnetic Ginzburg-Landau model. {\it Progress in Nonlinear Differential Equations and their
        Applications}, Vol. 70. Birkh\"{a}user, Basel (2007).
                \bibitem{SS-cmp} E. Sandier and S. Serfaty. From the Ginzburg-Landau model to vortex lattice problems. {\it Commun. Math. Phys.} {\bf 313} (3), 635-741 (2012).
        \end{thebibliography}
\end{document}